\documentclass[12pt,letterpaper,leqno]{amsart}
\usepackage{amsmath}
\usepackage[latin9]{inputenc}
\usepackage{color}
\usepackage{amsthm}
\usepackage{amstext}
\usepackage{amssymb}
\usepackage{latexsym}
\usepackage{amscd}
\usepackage{amsfonts}
\usepackage{enumerate}
\usepackage[mathscr]{euscript}

\makeatletter

\numberwithin{equation}{section}
\numberwithin{figure}{section}

\newlength{\baseunit}

\newcount{\numlines}
\newtheorem{theorem}{Theorem}
\newtheorem{claim}{Claim}

\newtheorem{example}{Example}
\newtheorem{lemma}{Lemma}
\newtheorem{algorithm}{Algorithm}

\DeclareMathOperator{\rt}{Root}

\newcommand{\LA}[1]{\refstepcounter{equation}\text{(\theequation)}\label{#1}}
\newcommand{\LAQ}[2]{\begin{itemize}\item[\LA{#1}]{#2} \end{itemize}}

\numberwithin{equation}{subsection}

\begin{document}
\title{Generators for the $C^m$-closures of Ideals}
\date{\today }
\author{Charles Fefferman, Garving K. Luli}
\thanks{CF is supported by NSF Grant DMS-1608782 and AFOSR Grant
FA9550-12-1-0425. GKL is supported by NSF Grant DMS-1554733.}
\maketitle


Let $\mathscr{R}$ denote the ring of real polynomials on $\mathbb{R}^{n}$.
Fix $m\geq 0$, and let $A_{1},\cdots ,A_{M}\in \mathscr{R}$. The \underline{$%
C^{m}$-closure} of $\left( A_{1},\cdots ,A_{M}\right) $, denoted here by $%
\left[ A_{1},\cdots ,A_{M};C^{m}\right] $, is the ideal of all $f\in 
\mathscr{R}$ expressible in the form $f=F_{1}A_{1}+\cdots +F_{M}A_{M}$ with
each $F_{i}\in C^{m}\left( \mathbb{R}^{n}\right) $. See \cite{cf-luli-dfq-cm}.

In this paper we exhibit an algorithm to compute generators for $\left[
A_{1},\cdots ,A_{M};C^{m}\right] $.

More generally, fix $m\geq 0$, and let $\boldsymbol{A=}\left( A_{ij}\right) 
_{\substack{ i=1,\cdots ,N \\ j=1,\cdots ,M}}$ be a matrix of (possibly discontinuous) semialgebraic
functions on $\mathbb{R}^{n}$.

We write $\left[ \boldsymbol{A};C^{m}\right] $ to denote the $\mathscr{R}$%
-module of all polynomial vectors $\vec{f}=\left( f_{1},\cdots ,f_{N}\right) 
$ (each $f_{i}\in \mathscr{R}$) expressible in the form 
\[
f_{i}\left( x\right) =\sum_{j=1}^{M}A_{ij}\left( x\right) F_{j}\left(
x\right) \text{ }\left( i=1,\cdots ,N\right) 
\]%
with each $F_{j}\in C^{m}\left( \mathbb{R}^{n}\right) $.

In this paper, we apply the main result of \cite{cf-luli-dfq-cm} to compute generators for $%
\left[ \boldsymbol{A};C^{m}\right] $.

Along the way, we provide an algorithm to compute generators for the ideal
of all polynomials that vanish on a given semialgebraic set $E\subset 
\mathbb{R}^{n}$. Another algorithm for this task appears in \cite{mohab}.

To understand $\left[ \boldsymbol{A};C^{m}\right] $, we study differential
operators $L$, acting on vectors of functions. Our operators $L$ have the
form 
\begin{equation}
L\vec{f}\left( x\right) =\sum_{\left\vert \alpha \right\vert \leq
s}\sum_{i=1}^{N}a_{\alpha }^{i}\left( x\right) \partial ^{\alpha
}f_{i}\left( x\right) \text{ for }\vec{f}=\left( f_{1},\cdots ,f_{N}\right) 
\text{,}  \label{i-1}
\end{equation}%
where the coefficients $a_{\alpha }^{i}$ are (possibly discontinuous)
semialgebraic functions on $\mathbb{R}^n$. We call an operator of the form (\ref{i-1}) a 
\underline{semialgebraic differential operator}.

Given a semialgebraic differential operator $L$, we introduce the $\mathscr{R}$%
-module $\mathcal{M}\left( L\right) $, consisting of all polynomial vectors $%
\vec{P}=\left( P_{1},\cdots ,P_{N}\right) $ (each $P_{i}\in \mathscr{R}$)
such that $L\left( Q\vec{P}\right) =0$ on $\mathbb{R}^{n}$ for all $Q\in 
\mathbb{\mathscr{R}}$.

Our interest in $\mathcal{M}\left( L\right) $ arises from the following
result, proven in \cite{cf-luli-dfq-cm}. 

\begin{theorem}
\label{theorem-Ac}Given $m$ and $\boldsymbol{A}$, there exist semialgebraic
differential operators $L_{1},\cdots ,L_{K}$ such that $\left[ \boldsymbol{A}%
;C^{m}\right] =\mathcal{M}\left( L_{1}\right) \cap \cdots \cap \mathcal{M}%
\left( L_{K}\right) $. Moreover, the operators $L_{1},\cdots ,L_{K}$ can be
computed from $\boldsymbol{A}$ and $m$.
\end{theorem}

The main result of this paper is an algorithm to compute generators for $%
\mathcal{M}\left( L\right) $, given an arbitrary semialgebraic differential
operator $L$. Once we know generators for each $\mathcal{M}\left( L_{\nu
}\right) $ $\left( \nu =1,\cdots ,K\right) $ as in Theorem \ref{theorem-Ac},
standard computational algebra \cite{grobner,cox} allows us to compute generators for $%
\mathcal{M}\left( L_{1}\right) \cap \cdots \cap \mathcal{M}\left( L_{K
}\right) $. Thanks to Theorem \ref{theorem-Ac}, the task of exhibiting
generators for $\left[ \boldsymbol{A};C^{m}\right] $ is thus reduced to the
task of computing generators for $\mathcal{M}\left( L\right) $.

The problem of computing the $C^{0}$-closure $\left[ \boldsymbol{A};C^{0}\right] $ was posed by Brenner \cite{Brenner}, and
Epstein-Hochster \cite{Hochster}, and solved by Fefferman-Koll\'ar \cite%
{Feff-Kollar} and Koll\'ar \cite{kollar}. See also \cite{cf-luli-dfq-cm}. So our results on  $\left[ \boldsymbol{A};C^{m}\right] $ are new only for $%
m\geq 1$.

To explain the ideas in our computation of $\mathcal{M}\left( L\right) $, we
consider in turn several examples of increasing complexity.

\begin{example}
\label{ex-i1}Let $Lf\left( x\right) =\mathbb{I}_{E}\left( x\right) f\left(
x\right) $, where $E\subset \mathbb{R}^{n}$ is semialgebraic. and $\mathbb{I}%
_{E}$ denotes the indicator function of $E$.
\end{example}

Thus, $L$ is a $0^{th}$ order operator acting on scalar functions $f$.

Then $\mathcal{M}\left( L\right) $ is the ideal $I\left( E\right) $ of all
polynomials that vanish on $E$. To get a hint of the issues that arise, let $%
E_{a}=\left\{ \left( x,y,z\right) \in \mathbb{R}^{3}:x^{2}-zy^{2}=0,z\leq
a\right\} $ for all $a\in \mathbb{R}$.

If $a>0$, one checks that $I\left( E_{a}\right) $ is the principal ideal
generated by $x^{2}-zy^{2}$. However, if $a < 0$, then $E_{a}=\left\{
\left( 0,0,z\right) :z\leq a\right\} $ and $I\left( E_{a}\right) $ is the
ideal generated by $x$ and $y$.

Remarkably, there are no standard algorithms to compute generators for $%
I\left( E\right) $ given a semialgebraic set $E$. Safey El Din et al \cite{mohab} produce an
algorithm that does the job. Here, we introduce a simpler but less efficient
algorithm than that of \cite{mohab} to compute generators for $I\left( E\right) $,
using a geometric idea to reduce matters to standard algorithms.

Our algorithm for $I\left( E\right) $ proceeds as follows. First, we
partition $E$ into finitely many simple smooth pieces $E_{\nu }$ $\left( \nu
=1,\cdots ,\nu _{\max }\right) $.

To each $E_{\nu }\subset \mathbb{R}^{n}$ we associate a complex variety $%
V_{\nu }\mathbb{\subset \mathbb{C}}^{n}$ (the \textquotedblleft
complexification" of $E_{\nu }$) and show that the polynomials vanishing on $%
E_{\nu }$ are precisely those that vanish on $V_{\nu }$. 

This reduces the computation of generators for $I\left( E_{\nu }\right) $ to
the corresponding problem for complex affine varieties, which is
well-understood (See \cite{grobner,cox}). Because $E$ is the union of the $E_{\nu }$, we have $%
I\left( E\right) =I\left( E_{1}\right) \cap \cdots \cap I\left( E_{\nu
_{\max }}\right) $. Once we know generators for each $I\left( E_{\nu
}\right) $, standard algorithms \cite{grobner,cox} produce generators for their
intersection. So we can compute generators for $\mathcal{M}\left( L\right) $
in Example \ref{ex-i1}.

\begin{example}
\label{ex-i2}Suppose our operator $L$ has polynomial coefficients, i.e., the 
$a_{\alpha }^{i}$ in $\left( \ref{i-1}\right) $ are polynomials.
\end{example}

Then we can easily prove the following assertion.

\begin{claim}
\label{claim-i1}$\mathcal{M}\left( L\right) $ consists of all polynomial
vectors $\left( P_{1},\cdots ,P_{N}\right) $ that solve a system of linear
equations 
\begin{equation}
\sum_{i=1}^{N}A_{\nu }^{i}\left( x\right) P_{i}\left( x\right) =0\text{ }%
\left( \nu =1,\cdots ,\nu _{\max }\right)   \label{i-2}
\end{equation}%
with polynomial coefficients $A_{\nu }^{i}$. Moreover, the $A_{\nu }^{i}$
may be computed from $L$.
\end{claim}

Standard computational algebra \cite{grobner,cox} produces generators for the $\mathscr{R}$%
-module of polynomial vectors satisfying (\ref{i-2}).

To prove Claim \ref{claim-i1} we proceed by induction on $s$, the order of
the operator $L$. In the base case $s=0$ there is nothing to prove. 

For the induction step, we fix $s\geq 1$, assume Claim \ref{claim-i1} for
operators of order less than $s$, and suppose $L$ is given by (\ref{i-1}),
with polynomial coefficients. We can then write $L$ in the form 
\begin{equation}
L\left( f_{1},\cdots ,f_{N}\right) \left( x\right) =\sum_{\left\vert \alpha
\right\vert = s}\partial ^{\alpha }\left\{ \sum_{i=1}^{N}a_{\alpha
}^{i}f_{i}\right\} \left( x\right) +\tilde{L}\left( f_{1},\cdots
,f_{N}\right) \left( x\right) \text{,}  \label{i-3}
\end{equation}%
where $\tilde{L}$ has polynomial coefficients and order less than $s$.

By definition, a polynomial vector $\vec{P}=\left( P_{1},\cdots
,P_{N}\right) $ belongs to $\mathcal{M}\left( L\right) $ if and only if $%
L\left( Q\vec{P}\right) =0$ on $\mathbb{R}^{n}$ for every polynomial $Q$.

For fixed $x_{0}\in \mathbb{R}^{n}$ and fixed multiindex $\gamma $ of order $%
\left\vert \gamma \right\vert =s$, we take $Q\left( x\right) =\frac{1}{\gamma
!}\left( x-x_{0}\right) ^{\gamma }$. Substituting into (\ref{i-3}), we see
that 
\[
L\left( Q\vec{P}\right) \left( x_{0}\right) =\sum_{i=1}^{N}a_{\gamma
}^{i}\left( x_{0}\right) \cdot P_{i}\left( x_{0}\right) \text{.}
\]%
Therefore, if $\vec{P}=\left( P_{1},\cdots ,P_{N}\right) $ belongs to $%
\mathcal{M}\left( L\right) $, then 
\begin{equation}
\sum_{i=1}^{N}a_{\gamma }^{i}P_{i}=0\text{ on }\mathbb{R}^{n}\text{ for }%
\left\vert \gamma \right\vert =s;  \label{i-4}
\end{equation}%
consequently, for any $Q\in \mathscr{R}$ we have 
\[
0=L\left( Q\vec{P}\right) =\sum_{\left\vert \alpha \right\vert =
s}\partial ^{\alpha }\left\{ Q\sum_{i=1}^{N}a_{\alpha }^{i}P_{i}\right\} +%
\tilde{L}\left( Q\vec{P}\right) =\tilde{L}\left( Q\vec{P}\right) \text{.}
\]%
Thus, if $\vec{P}$ belongs to $\mathcal{M}\left( L\right) $, then $\left( %
\ref{i-4}\right) $ holds, and 
\begin{equation}
\vec{P}\in \mathcal{M}\left( \tilde{L}\right) \text{.}  \label{i-5}
\end{equation}

Conversely,  $\left( \ref{i-4}\right) $ and  $\left( \ref{i-5}\right) $
obviously imply $\vec{P}\in \mathcal{M}\left( L\right) $.

So $\mathcal{M}\left( L\right) $ consists of all $\vec{P}\in \mathcal{M}%
\left( \tilde{L}\right) $ satisfying $\left( \ref{i-4}\right) $.

Claim \ref{claim-i1} for $L$ now follows at once from Claim \ref{claim-i1}
for $\tilde{L}$, completing our induction on $s$.

So we can compute generators for $\mathcal{M}\left( L\right) $ in Example %
\ref{ex-i2}.

\begin{example}
\label{ex-i3} (Generalizes Example \ref{ex-i2}.) Let $\Gamma =\left\{ \left(
x,G\left( x\right) \right) :x\in U\right\} $, where $U\subset \mathbb{R}^{n}$
 is an open semialgebraic set and $G:U\rightarrow \mathbb{R}^{p}is$ a
real-analytic semialgebraic function. We write $\left( x,y\right) $ to
denote a point of $\mathbb{R}^{n}\times \mathbb{R}^{p}$. For this example, $%
\mathscr{R}$ denotes the ring of polynomials on $\mathbb{R}^{n}\times 
\mathbb{R}^{p}$. Let $L$ be a differential operator of the form 
\begin{equation}
L\left( f_{1},\cdots ,f_{N}\right) \left( x,y\right) =\sum_{\left\vert
\alpha \right\vert ,\left\vert \beta \right\vert \leq
s}\sum_{i=1}^{N}a_{\alpha \beta }^{i}\left( x,y\right) \partial _{x}^{\alpha
}\partial _{y}^{\beta }f_{i}\left( x,y\right) \label{i-6}
\end{equation}

with polynomial coefficients $a_{\alpha \beta }^{i}\left( x,y\right) $, and
let 
\[
L^{\Gamma }\vec{f}\left( x,y\right) =\mathbb{I}_{\Gamma }\left( x,y\right)
\cdot L\vec{f}\left( x,y\right) \text{.}
\]
\end{example}

Thus, $\mathcal{M}\left( L^{\Gamma }\right) $ is the $\mathscr{R}$-module of
all polynomial vectors $\vec{P}=\left( P_{1},\cdots ,P_{N}\right) $ (each $%
P_{i}\in \mathscr{R}$), such that 
\begin{equation}
L\left( Q\vec{P}\right) =0\text{ on }\Gamma \text{ for all }Q\in \mathscr{R}%
\text{.}  \label{i-7}
\end{equation}

We write $\mathcal{M}\left( L,\Gamma \right) $ to denote the $\mathscr{R}$%
-module of all $\vec{P}$ that satisfy (\ref{i-7}).

This generalizes Example \ref{ex-i2}, which arises here as the case $p=0$.

The analogue of Claim \ref{claim-i1} here is as follows.

\begin{claim}
\label{claim-i2}There exist differential operators $L_{1},\cdots ,L_{\nu
_{\max }}$, of the form 
\[
L_{\nu }\left( f_{1},\cdots ,f_{N}\right) \left( x,y\right)
=\sum_{\left\vert \beta \right\vert \leq s}\sum_{i=1}^{N}b_{\nu \beta
}^{i}\left( x,y\right) \partial _{y}^{\beta }f_{i}\left( x,y\right) \text{ }(%
\text{each }b_{\nu \beta }^{i}\in \mathscr{R}),
\]%
such that 
\[
\mathcal{M}\left( L^{\Gamma }\right) =\mathcal{M}\left( L_{1},\Gamma \right)
\cap \cdots \cap \mathcal{M}\left( L_{\nu _{\max }},\Gamma \right) \text{.}
\]%
Moreover, we can compute the $L_{\nu }$ from $L$ and $\Gamma $.
\end{claim}

The point is that the $L_{\nu }$ involve no $x$-derivatives.

By using our algorithm to compute $I\left( E\right) $ (see Example \ref%
{ex-i1}), we can reduce the computation of generators for each $\mathcal{M}%
\left( L_{\nu },\Gamma \right) $ to the study of a system of linear
equations with polynomial coefficients, which can be treated by standard
computational algebra. Thus, we are able to produce generators for $\mathcal{%
M}\left( L^{\Gamma }\right) $ in Example \ref{ex-i3}.

\begin{example}
\label{ex-i4}We generalize Example \ref{ex-i3}. We suppose that the
coefficients $a_{\alpha \beta }^{i}$ in $\left( \ref{i-6}\right) $ are
real-analytic semialgebraic functions, rather than polynomials.
\end{example}

We reduce matters to Example \ref{ex-i3} by introducing new variables $%
z_{\alpha \beta }^{i}$ to replace the coefficients $a_{\alpha \beta }^{i}$.
We write $z$ to denote $\left( z_{\alpha \beta }^{i}\right) _{\substack{ %
i=1,\cdots ,N \\ \left\vert \alpha \right\vert ,\left\vert \beta \right\vert
\leq s}}$ .

Instead of $\Gamma =\left\{ \left( x,G\left( x\right) \right) :x\in
U\right\} $, we consider 
\[
\Gamma ^{+}=\left\{ \left( x,y,z\right) :\left( x,y\right) \in \Gamma
,z_{\alpha \beta }^{i}=a_{\alpha \beta }^{i}\left( x,y\right) \right\} \text{%
;}
\]%
and instead of $L$, we consider 
\[
L^{+}\left( f_{1},\cdots ,f_{N}\right) \left( x,y,z\right) =\sum_{\left\vert
\alpha \right\vert ,\left\vert \beta \right\vert \leq
s}\sum_{i=1}^{N}z_{\alpha \beta }^{i}\partial _{x}^{\alpha }\partial
_{y}^{\beta }f_{i}\left( x,y,z\right) \text{.}
\]

The algorithm for Example \ref{ex-i3} then produces generators for the
module of all polynomial vectors $\vec{P}$ in $\left( x,y,z\right) $ such
that 
\begin{equation}
L^{+}\left( Q\vec{P}\right) =0\text{ on }\Gamma ^{+}\text{ for all }Q\text{.}
\label{i-8}
\end{equation}

We then have to understand which solutions $\vec{P}$ of \eqref{i-8} do not
depend on $z$. We carry this out in Section \ref{section7} below, thus producing
generators for the module $\mathcal{M}\left( L^{\Gamma }\right) $ in Example %
\ref{ex-i4}.

Finally, we compute generators for $\mathcal{M}\left( L\right) $ in the
general case. Let $L$ be a semialgebaic differential operator. Standard
algorithms allow us to partition $\mathbb{R}^{n}$ into semialgebraic sets $%
E_{\nu }$ $\left( \nu =1,\cdots ,\nu _{\max }\right) $ on each of which
(after a $\nu $-dependent linear change of coordinates) $L$ and $E_{\nu }$ may be
brought to the form of Example \ref{ex-i4}, with $E_{\nu }$ playing the r%
\^{o}le of $\Gamma $. Therefore, $\mathcal{M}\left( L\right) $ is the
intersection of finitely many $\mathscr{R}$-modules, for each of which we
can produce generators. The computation of generators for $\mathcal{M}\left(
L\right) $ may then be accomplished by standard computational algebra.

This concludes the introductory explanation of our algorithm to compute $%
\mathcal{M}\left( L\right) $. For full details, see Sections \ref{preliminary}-\ref{section7} below.

We have made no attempt here to estimate the number of computer operations
needed to execute our algorithms. Surely an expert in computational
semialgebraic geometry could make significant improvements in our
algorithms, and estimate the complexity of the improved algorithms. We
welcome such progress. 

We are grateful to Matthias Aschenbrenner, Saugata Basu, Edward Bierstone, Jes\'{u}s De Loera,
Zeev Dvir, J\'anos Koll\'ar,
Pierre Milman, Wieslaw Paw\l {}ucki, Mohab Safey El Din, Ary Shaviv, Rekha Thomas, and the participants in the 9%
$^{th}$-11$^{th}$ Whitney workshops for valuable discussions, and to the
Technion -- Israel Institute of Technology, College of William and Mary, and
Trinity College Dublin, for hosting the above workshops. 

\section{Preliminaries}\label{preliminary}
We begin with a few elementary lemmas.

\begin{lemma}
\label{prelim-lemma1}Suppose $V\subset \mathbb{R}^{n}\times \mathbb{R}^{m}$
is a connected, real-analytic submanifold. Write $\left( x_{1},\cdots
,x_{n},y_{1},\cdots ,y_{m}\right) $ to denote a point of $\mathbb{R}%
^{n}\times \mathbb{R}^{m}$, and suppose that $x_{1},\cdots ,x_{n}$ are local
coordinates on a nonempty relatively open set $U\subset V$. Let $\Delta
\left( x_{1},\cdots ,x_{n}\right) $ be a nonzero polynomial.
Let $P\left( x_{1},\cdots ,x_{n},y_{1},\cdots ,y_{m}\right) $ be a
polynomial. If $\Delta \left( x_{1},\cdots ,x_{n}\right) \cdot P\left(
x_{1},\cdots ,x_{n},y_{1},\cdots ,y_{m}\right) =0$ on $V$, then $P\left(
x_{1},\cdots ,x_{n},y_{1},\cdots ,y_{m}\right) =0$ on $V$.
\end{lemma}

\begin{proof}
Since $\left( x_{1},\cdots ,x_{n}\right) $ are local coordinates on $U$, we
know that 
\begin{equation*}
\left\{ \left( x_{1},\cdots ,x_{n},y_{1},\cdots ,y_{m}\right) \in U:\Delta
\left( x_{1},\cdots ,x_{n}\right) \not=0\right\}
\end{equation*}
is dense in $U$, hence $P\left( x_{1},\cdots ,x_{n},y_{1},\cdots
,y_{m}\right) =0$ on $U$. The lemma follows, since $P$ is a polynomial and $%
V $ is a connected real-analytic manifold.
\end{proof}

\begin{lemma}
\label{prelim-lemma2}For each $\mu =1,\cdots ,m$, let 
\begin{equation*}
\tilde{P}_{\mu }\left( x_{1},\cdots ,x_{n},y_{\mu }\right) =a_{\mu }\left(
x_{1},\cdots ,x_{n}\right) y_{\mu }^{D_{\mu }}+\sum_{j<D_{\mu }}b_{j\mu
}\left( x_{1},\cdots ,x_{n}\right) \left( y_{\mu }\right) ^{j}
\end{equation*}%
be a polynomial in $\mathbb{R}\left[ x_{1},\cdots ,x_{n},y_{\mu }\right] $,
with $a_{\mu }\left( x_{1},\cdots ,x_{n}\right) $ nonzero. Let $%
\Delta \left( x\right) =\prod_{\mu =1}^{m}a_{\mu }\left( x_{1},\cdots
,x_{n}\right) $ for $x=\left( x_{1},\cdots ,x_{n}\right) $. Let $P\left(
x_{1},\cdots ,x_{n},y_{1},\cdots ,y_{m}\right) \in \mathbb{R}\left[
x_{1},\cdots ,x_{n},y_{1},\cdots ,y_{m}\right] $ be given and let $K\geq 1$.
Then, for some $l\geq 0$, we can express 
\begin{eqnarray*}
\left( \Delta \left( x\right) \right) ^{l}P\left( x_{1},\cdots
,x_{n},y_{1},\cdots ,y_{m}\right)  &=&\sum_{\mu =1}^{m}H_{\mu }\left(
x_{1},\cdots ,x_{n},y_{1},\cdots ,y_{m}\right) \cdot \left[ \tilde{P}_{\mu
}\left( x_{1},\cdots ,x_{n},y_{\mu }\right) \right] ^{K} \\
&&+P^{\#}\left( x_{1},\cdots ,x_{n},y_{1},\cdots ,y_{m}\right) ,
\end{eqnarray*}%
where $H_{\mu }$ and $P^{\#}$ are polynomials, and $\deg _{y_{1},\cdots
,y_{m}}P^{\#}\leq D$; here, $D$ is a constant that may be computed from the $%
D_{\mu }$ and $K$.
\end{lemma}

\begin{proof}
Let $\mathscr{R}$ be the ring of rational functions of the form 
\begin{equation*}
\frac{P\left( x_{1},\cdots ,x_{n}\right) }{\left( \Delta \left( x\right)
\right) ^{p}}\text{, for some polynomial }P\text{ and some integer power }p%
\text{.}
\end{equation*}%
Each $\tilde{P}_{\mu }$ is a unit times a monic polynomial in $y_{\mu }$,
when we work in the ring $\mathscr{R}\left[ y_{1},\cdots ,y_{m}\right] $.
Hence, in $\mathscr{R}\left[ y_{1},\cdots ,y_{m}\right] $, we may divide
polynomials by powers of $\tilde{P}_{\mu }$, obtaining a quotient and a
remainder in the usual way.

By induction on $i\geq 0$ (with $i\leq m$), we show that we can write $P$ in
the form 
\begin{eqnarray}
&&P\left( x_{1},\cdots ,x_{n},y_{1},\cdots ,y_{m}\right)   \label{prelim-1}
\\
&=&\sum_{1\leq \mu \leq i}H_{\mu }\left(y_{1},\cdots
,y_{m}\right) \left[ \tilde{P}_{\mu }\left( x_{1},\cdots ,x_{n},y_{\mu
}\right) \right] ^{K}+R\left( y_{1},\cdots ,y_{m}\right) 
\notag
\end{eqnarray}%
with $H_{\mu },R\in \mathscr{R}\left[ y_{1},\cdots ,y_{m}\right] $ and $\deg
_{y_{\mu }}R<KD_{\mu }$ for $1\leq \mu \leq i$.

In the base case, $i=0$, so (\ref{prelim-1}) is trivial. We take $R=P$.

Suppose (\ref{prelim-1}) holds for a given $i<m$. We will prove that $P$
may be expressed in the form (\ref{prelim-1}), with $i+1$ in place of $i$.

Write 
\begin{equation*}
R\left(y_{1},\cdots ,y_{m}\right) =\sum_{\alpha
_{1}<KD_{1}}\cdots \sum_{\alpha _{i}<KD_{i}}R_{\left( \alpha _{1},\cdots
,\alpha _{i}\right) }\left( y_{i+1},\cdots ,y_{m}\right)
y_{1}^{\alpha _{1}}\cdots y_{i}^{\alpha _{i}},
\end{equation*}%
with each $R_{(\alpha_1,\cdots, \alpha_i)} \in \mathscr{R}[y_{i+1},\cdots,y_m]$.
Performing a division by $\left[ \tilde{P}_{i+1}\left( x_{1},\cdots
,x_{n},y_{i+1}\right) \right] ^{K}$ in the ring $\mathscr{R}\left[y_{i+1},\cdots ,y_{m}\right] $, we may write 
\begin{equation*}
R_{\left( \alpha _{1},\cdots ,\alpha _{i}\right) }=\tilde{H}_{\left( \alpha
_{1},\cdots ,\alpha _{i}\right) }\cdot \left[ \tilde{P}_{i+1}\left( x_{1},\cdots
,x_{n},y_{i+1}\right) \right] ^{K}+S_{\left( \alpha _{1},\cdots ,\alpha
_{i}\right) }\left( x_{1},\cdots ,x_{n},y_{i+1},\cdots ,y_{m}\right) 
\end{equation*}%
with 
\begin{equation*}
\tilde{H}_{\left( \alpha _{1},\cdots ,\alpha _{i}\right) } ,S_{\left( \alpha
_{1},\cdots ,\alpha _{i}\right) }\in \mathscr{R}\left[y_{i+1},\cdots ,y_{m}\right] 
\end{equation*}%
and 
\begin{equation*}
\deg _{y_{i+1}}S_{\left( \alpha _{1},\cdots ,\alpha _{i}\right) }<KD_{i+1}.
\end{equation*}%
Now set 
\begin{equation*}
\tilde{H}=\sum_{\alpha _{1}<KD_{1}}\cdots \sum_{\alpha _{i}<KD_{i}}\tilde{H}%
_{\left( \alpha _{1},\cdots ,\alpha _{i}\right) }y_{1}^{\alpha _{1}}\cdots
y_{i}^{\alpha _{i}}
\end{equation*}%
and 
\begin{equation*}
S=\sum_{\alpha _{1}<KD_{1}}\cdots \sum_{\alpha _{i}<KD_{i}}S_{\left( \alpha
_{1},\cdots ,\alpha _{i}\right) }y_{1}^{\alpha _{1}}\cdots y_{i}^{\alpha
_{i}}\text{.}
\end{equation*}%
Then $\tilde{H}$, $S\in \mathscr{R}\left[ y_{1},\cdots
,y_{m}\right] $ and $\deg _{y_{\mu }}S<KD_{\mu }$ for $\mu =1,\cdots ,i+1$.
Moreover, $R=\tilde{H}\cdot \left[ \tilde{P}_{i+1}\right] ^{K}+S$. Therefore, by (%
\ref{prelim-1}), we have 
\begin{eqnarray*}
&&P\left( x_{1},\cdots ,x_{n},y_{1},\cdots ,y_{m}\right)  \\
&=&\sum_{1\leq \mu \leq i}H_{\mu }\left( x_{1},\cdots ,x_{n},y_{1},\cdots
,y_{m}\right) \cdot \left[ \tilde{P}_{\mu }\left( x_{1},\cdots
,x_{n},y_{\mu} \right) \right] ^{K} \\
&&+\tilde{H}\left( y_{1},\cdots ,y_{m}\right) \cdot %
\left[ \tilde{P}_{i+1}\left(x_1, \cdots, x_n, y_{i+1}\right) \right]
^{K} \\
&&+S\left( y_{1},\cdots ,y_{m}\right) \text{,}
\end{eqnarray*}%
proving (\ref{prelim-1}) with $i+1$ in place of $i$.

This completes our induction, and establishes (\ref{prelim-1}).

Now taking $i=m$ in (\ref{prelim-1}), and clearing denominators by
multiplying by a high power of $\left( \Delta \left( x\right) \right) $, we
obtain the conclusion of Lemma \ref{prelim-lemma2}.
\end{proof}

\begin{lemma}
\label{prelim-lemma3}Let $\mathscr{R}$ be a ring, and let 
\begin{equation*}
\tilde{P}_{\mu }\left( y_{\mu }\right) =y_{\mu }^{D_{\mu
}}+\sum_{i=1}^{D_{\mu }-1}a_{\mu i}y_{\mu }^{i}
\end{equation*}%
be monic polynomials with coefficients in $\mathscr{R}$ for $\mu =1,\cdots ,m
$.

Let $H_{\mu }(y_{1},\cdots ,y_{m})$ be polynomials with coefficients in $%
\mathscr{R}$, and suppose that 
\begin{equation*}
\sum_{\mu =1}^{m}H_{\mu }\left( y_{1},\cdots ,y_{m}\right) \tilde{P}_{\mu
}\left( y_{\mu }\right) 
\end{equation*}%
has degree $\leq D$, where $D\geq \max \left\{ D_{1},\cdots ,D_{m}\right\} $.

Then there exist polynomials $\tilde{H}_{\mu }\left( y_{1},\cdots
,y_{m}\right) $ with coefficients in $\mathscr{R}$ such that 
\begin{equation}
\sum_{\mu =1}^{m}\tilde{H}_{\mu }\left( y_{1},\cdots ,y_{m}\right) \tilde{P}%
_{\mu }\left( y_{\mu }\right) =\sum_{\mu =1}^{m}H_{\mu }\left( y_{1},\cdots
,y_{m}\right) \tilde{P}_{\mu }\left( y_{\mu }\right)   \label{prelim-*}
\end{equation}%
with $\deg \tilde{H}_{\mu }\leq D-D_{\mu }$ for each $\mu $.
\end{lemma}

\begin{proof}
Let $M=\max_{1\leq \mu \leq m}\left\{ D_{\mu }+\deg H_{\mu }\right\} $. It
is enough to show that, if $M>D$, then there exist $\tilde{H}_{\mu }$
satisfying (\ref{prelim-*}) such that 
\begin{equation*}
\max_{1\leq \mu \leq m}\left\{ D_{\mu }+\deg \tilde{H}_{\mu }\right\} <M.
\end{equation*}%
To show this, we proceed as follows.

Let $H_{\mu }=H_{\mu }^{o}+H_{\mu }^{err}$, where $H_{\mu }^{o}$ is
homogeneous of degree $M-D_{\mu }$, and $\deg H_{\mu }^{err}<M-D_{\mu }$.

If $M>D$, then $\sum_{\mu }H_{\mu }^{o}y_{\mu }^{D_{\mu }}=0$, since $H_{\mu
}\tilde{P}_{\mu }=H_{\mu }^{o}y_{\mu }^{D_{\mu }}+$(terms of degree $<M$).

Writing $H_{\mu }^{o}=\sum_{\left\vert \beta \right\vert =M-D_{\mu
}}A_{\beta }^{\mu }y^{\beta }$ with $A_{\beta }^{\mu }\in \mathscr{R}$, and
defining 
\begin{equation*}
S\left( \gamma \right) =\left\{ i:\gamma _{i}\geq D_{i}\right\}
\end{equation*}%
for multiindices $\gamma =\left( \gamma _{1},\cdots ,\gamma _{m}\right) $,
we conclude that 
\begin{equation*}
\sum_{\mu \in S\left( \gamma \right) }A_{\gamma -D_{\mu }\mathbb{I}_{\mu
}}^{\mu }=0
\end{equation*}%
for each $\left\vert \gamma \right\vert =M$. Here $\mathbb{I}_{\mu }=\left(
a_{1},\cdots ,a_{m}\right) $ with $a_{\mu }=1$ and $a_{i}=0$ when $i\not=\mu 
$.

We now use the following observation: Let $A_{\mu }\in \mathscr{R}$ for $\mu
\in S$ (a finite set), and suppose that $\sum_{\mu \in S}A_{\mu }=0$. Then
there exist $\lambda _{ij}\in \mathscr{R}$ for $i,j\in S$ distinct, such
that 
\begin{equation*}
A_{\mu }=\sum_{i,j\in S\text{,}i\not=j}\lambda _{ij}\left( \delta _{\mu
i}-\delta _{\mu j}\right) \text{ for each } \mu \in S.
\end{equation*}%
Applying the observation to the $A_{\gamma -D_{\mu }\mathbb{I}_{\mu }}^{\mu
} $ for fixed $\gamma $, we obtain $\lambda _{ij}^{\gamma }\in \mathscr{R}$
for $i,j\in S\left( \gamma \right) $ distinct, $\left\vert \gamma
\right\vert =M$, such that 
\begin{equation*}
A_{\gamma -D_{\mu }\mathbb{I}_{\mu }}^{\mu }=\sum_{i,j\in S(\gamma)\text{,}%
i\not=j}\lambda _{ij}^{\gamma }\left( \delta _{\mu i}-\delta _{\mu j}\right)
\text{ for } \mu \in S(\gamma).
\end{equation*}%
Note that for $i,j\in S\left( \gamma \right) $ distinct, $\gamma -D_{i}%
\mathbb{I}_{i}-D_{j}\mathbb{I}_{j}$ is a multiindex (i.e., its components
are non-negative).

For $\left\vert \gamma \right\vert =M$, let $H_{\mu }^{\gamma }=\sum_{i,j\in
S\left( \gamma \right) ,i\not=j}\lambda _{ij}^{\gamma }y^{\gamma -D_{i}%
\mathbb{I}_{i}-D_{j}\mathbb{I}_{j}}\left( \tilde{P}_{j}\left( y_{j}\right) \delta
_{\mu i}-\tilde{P}_{i}\left( y_{i}\right) \delta _{\mu j}\right) $.

Note that 
\begin{eqnarray*}
&&\sum_{\mu }H_{\mu }^{\gamma }\tilde{P}_{\mu }\left( y_{\mu }\right)  \\
&=&\sum_{i,j\in S\left( \gamma \right) ,i\not=j}\lambda _{ij}^{\gamma
}y^{\gamma -D_{i}\mathbb{I}_{i}-D_{j}\mathbb{I}_{j}}\left( \tilde{P}_{j}\left(
y_{j}\right) \sum_{\mu }\delta _{\mu i}\tilde{P}_{\mu }\left( y_{\mu
}\right) -\tilde{P}_{i}\left( y_{i}\right) \sum_{\mu }\delta _{\mu j}\tilde{P}_{\mu
}\left( y_{\mu }\right) \right)  \\
&=&0.
\end{eqnarray*}

Moreover, 
\begin{eqnarray*}
&&H_{\mu }^{\gamma }\tilde{P}_{\mu }\left( y_{\mu }\right)  \\
&=&\sum_{i,j\in S\left( \gamma \right) ,i\not=j}\lambda _{ij}^{\gamma
}y^{\gamma }\left( \delta _{\mu i}-\delta _{\mu j}\right) +\text{(terms of
degree }<M\text{)} \\
&=& \begin{cases} A_{\gamma -D_{\mu }\mathbb{I}_{\mu }}^{\mu }y^{\gamma }+\text{(terms of
degree }<M)\text{ if } S(\gamma) \ni \mu \\
\text{(terms of degree } <M) \text{ if } S(\gamma) \not \ni \mu 
\end{cases}
\end{eqnarray*}%
whereas 
\begin{eqnarray*}
&&H_{\mu }\tilde{P}_{\mu }\left( y_{\mu }\right)  \\
&=&\sum_{\left\vert \gamma \right\vert =M, S(\gamma) \ni \mu }A_{\gamma -D_{\mu }\mathbb{I}_{\mu }}^{\mu }
y^{\gamma }+\text{(terms of degree }<M\text{).}
\end{eqnarray*}%
Setting $\tilde{H}_{\mu }=H_{\mu }-\sum_{\left\vert \gamma \right\vert
=M}H_{\mu }^{\gamma }$, we conclude that 
\begin{equation*}
\deg \left( \tilde{H}_{\mu }\tilde{P}_{\mu }\left( y_{\mu }\right) \right) <M%
\text{ (all }\mu \text{), and that }\sum_{\mu }\tilde{H}_{\mu }\tilde{P}%
_{\mu }\left( y_{\mu }\right) =\sum_{\mu }H_{\mu }\tilde{P}_{\mu }\left(
y_{\mu }\right) \text{.}
\end{equation*}%
The proof of the lemma is complete.
\end{proof}

\begin{lemma}
\label{prelim-lemma4}For $\mu =1,\cdots ,m$, let 
\begin{equation*}
\tilde{P}_{\mu }\left( x_{1},\cdots ,x_{n},y_{\mu }\right) =a_{\mu }\left(
x_{1},\cdots ,x_{n}\right) y_{\mu }^{D_{\mu }}+\sum_{i=0}^{D_{\mu }-1}a_{\mu
i}\left( x_{1},\cdots ,x_{n}\right) \left( y_{\mu }\right) ^{i}
\end{equation*}%
be a polynomial with $a_{\mu }$ nonzero.

Let $\Delta (x_{1},\cdots ,x_{n})=\prod_{\mu =1}^{m}a_{\mu }(x_{1},\cdots
,x_{n})$. Let $H_{\mu }(x_{1},\cdots ,x_{n},y_{1},\cdots ,y_{m})$ be
polynomials ($\mu =1,\cdots ,m$), and suppose that 
\begin{equation*}
\deg _{y}\left\{ \sum_{\mu =1}^{m}H_{\mu }\left( x_{1},\cdots
,x_{n},y_{1},\cdots ,y_{m}\right) \tilde{P}_{\mu }\left( x_{1},\cdots
,x_{n},y_{\mu }\right) \right\} \leq D,
\end{equation*}%
where $D\geq \max \left\{ D_{1},\cdots ,D_{m}\right\} $.

Then there exist polynomials $H_\mu^{\#}$ such that for some $l$ we have

\begin{itemize}
\item \quad $\left( \Delta \left( x\right) \right) ^{l}\sum_{\mu
=1}^{m}H_{\mu }\left( x_{1},\cdots ,x_{n},y_{1},\cdots ,y_{m}\right) \tilde{P%
}_{\mu }\left( x_{1},\cdots ,x_{n},y_{\mu }\right) $ \newline
$=\sum_{\mu =1}^{m}H_{\mu }^{\#}\left( x_{1},\cdots ,x_{n},y_{1},\cdots
,y_{m}\right) \tilde{P}_{\mu }\left( x_{1},\cdots ,x_{n},y_{\mu }\right) $ 
\newline
and

\item $\deg _{y}H_{\mu }^{\#}\leq D-D_{\mu }$ for each $\mu =1,\cdots ,m$.
\end{itemize}
\end{lemma}

\begin{proof}
Immediate from Lemma \ref{prelim-lemma3} with $\mathscr{R}$ taken to be the
ring of rational functions of the form $\frac{P\left( x_{1},\cdots
,x_{n}\right) }{\left( \Delta \left( x\right) \right) ^{p}}$, for some
polynomial $P$ and some integer power $p$.
\end{proof}

We recall a few elementary facts and definitions from algebraic geometry; see \cite{fulton,shafarevich}.

An \underline{algebraic set} in $\mathbb{C}^n$ is a set of the form $V=\{z \in \mathbb{C}^n:P_1(z)=\cdots=P_k(z)=0 \}$ where $P_1,\cdots,P_k$ are polynomials. $V$ is called \underline{irreducible} if it cannot be expressed as the union of two algebraic sets $V=V_1 \cup V_2$ with $V_1,V_2 \not=V$. An irreducible algebraic set is called an \underline{affine variety}. Every algebraic set $V \subset \mathbb{C}^n$ may be expressed as the union of finitely many affine varieties $V_1,\cdots, V_p$ with $V_i \not\subset V_j$ for $i\not= j$. These $V_i$ are the \underline{irreducible components} of $V$.
In any given affine variety $V$, the set $V_{\text{reg}}\subset V$ of \underline{regular points} of $V$ is a connected complex analytic  submanifold of $\mathbb{C}^n$. Moreover, $V_{\text{reg}}$ is dense in $V$.

\begin{lemma}\label{lemma5}
Suppose $W=\{z\in \mathbb{C}^n: P_\mu(z)=0 \text{ for } \mu=1,\cdots, m \}$ for polynomials $P_1,\cdots, P_m$. Let $V_1,\cdots, V_p$ be the irreducible components of $W$, and for each $j=1,\cdots, p$, let $V_{j,\text{reg}}$ be the set of all regular points of $V_j$.

Let $z_0 \in W$ and suppose that the differentials $dP_\nu(z_0)$ $(\nu=1,\cdots, m)$ are linearly independent.

Then there exist an index $j_0 \in \{1,\cdots, p\}$ and a small ball $B(z_0,r) \subset \mathbb{C}^n$ about $z_0$, such that $$W\cap B(z_0,r) = V_{j_0,\text{reg}} \cap B(z_0,r)$$ and $$V_j \cap B(z_0,r)=\emptyset \text{ for } j\not=j_0.$$

\end{lemma}

\begin{proof}
See \cite{shafarevich}.
\end{proof}

\section{Background from Computational Algebraic Geometry}

In this section, we present some known technology for computations
involving semialgebraic sets in $\mathbb{R}^n$ and algebraic sets in $\mathbb{C}^n$. See the reference book \cite{basu}.

We begin by describing our model of computation, copied from \cite{cf-luli-dfq-cm}. Our algorithms are to be
run on an idealized computer with standard von Neumann architecture \cite{vonneumann},
able to store and perform basic arithmetic operations on integers and
infinite precision real numbers, without roundoff errors or overflow
conditions. We suppose that our computer can access an ORACLE that solves
polynomial equations in one unknown. More precisely, the ORACLE answers
queries; a query consists of a non-constant polynomial $P$ (in one variable)
with real coefficients, and the ORACLE responds to a query $P$ by producing
a list of all the real roots of $P$.

Let us compare our model of computation with that of \cite{basu}.

All  arithmetic in \cite{basu} is performed within a subring $\Lambda$ of a real  
closed field $K$ (e.g. the integers sitting inside the reals). However,  
some algorithms in \cite{basu} produce as output a finite list of elements of  
$K$ not necessarily belonging to $\Lambda$. A field element $x_0$ arising in  
such an output is  specified by exhibiting a polynomial $P$ (in one  
variable) with coefficients in $\Lambda$ such that $P(x_0)=0$,  together with  
other data to distinguish $x_0$ from the other roots of $P$.

In our model of computation, we take $\Lambda$ and $K$ to consist of all real
numbers, and we query the ORACLE whenever \cite{basu} specifies a real number  
by means of a polynomial $P$ as above.

Next, we describe how we will represent a semialgebraic set $E$. We will
specify a Boolean combination of sets of the form

\LAQ{XXX}{$\left\{ \left(x_1, \cdots, x_n \right) \in \mathbb{R}^n:
P(x_1,\cdots, x_n)>0 \right\}$,}
\LAQ{YYY}{$\left\{ \left(x_1, \cdots, x_n \right) \in \mathbb{R}^n:
P(x_1,\cdots, x_n)<0 \right\}$, or}
\LAQ{ZZZ}{$\left\{ \left(x_1, \cdots, x_n \right) \in \mathbb{R}^n:
P(x_1,\cdots, x_n)=0 \right\}$}
for polynomials $P \in \mathbb{R}[x_1,\cdots, x_n]$.

A given semialgebraic set may be specified as above in many different ways, but that won't bother us.

A semialgebraic function $F: E \rightarrow \mathbb{R}^m$ is specified by specifying its graph $\{(x,F(x)): x \in E\} \subset E \times \mathbb{R}^m$. 

We will also make computations with algebraic sets $V \subset \mathbb{C}^n$. To specify $V$, we exhibit polynomials $P_1, \cdots, P_k \in \mathbb{C}[z_1,\cdots, z_n]$, such that 
$V =\{ z \in \mathbb{C}^n: P_1(z)=\cdots=P_k(z) =0 \}$. 

(We represent the coefficients of the $P_j$ in the obvious way, by specifying their real and imaginary parts.) 

Again, a given $V$ may be specified as above in many different ways, but that won't bother us.

\subsection{Known Algorithms\label{section-known-algorithms}}
In this subsection we present several known algorithms from computational algebraic geometry.

We begin with two algorithms that deal with algebraic sets in $\mathbb{C}^n$.

\begin{algorithm}\label{algorithm-generators-ideal-algebraic-set}
Given an algebraic set $V\subset\mathbb{C}^n$, we compute generators for the ideal of all $P \in \mathbb{C}[z_1,\cdots,z_n]$ that vanish on $V$. {\textup{(See \cite{becker}; see also \cite{eisenbud} for a different algorithm.)}} \end{algorithm}

\begin{algorithm}\label{connected-components}
Given an algebraic set $V\subset\mathbb{C}^n$, we compute its irreducible components $V_1, \cdots, V_p$. {\textup{(See \cite{fortuna}.)}}
\end{algorithm}

We will need several algorithms pertaining to semialgebraic sets $E \subset \mathbb{R}^n$.

\begin{algorithm}\label{dimension-semialgebraic-alg}
Given a semialgebraic set $E$, we compute its dimension. \textup{(See Algorithm 14.31 in \cite{basu}.)}
\end{algorithm}

\begin{algorithm}\label{Algorithm-computing-connected-component} Given a semialgebraic set $E$, we compute the connected components of $E$. In particular, if $E$ is zero-dimensional (and therefore finite), we compute a list of all the points of $E$. \textup{(See Algorithm 16.20 in \cite{basu}.)}
\end{algorithm}

\begin{algorithm}\label{algorithm-checking-projection}
Given semialgebraic sets $E_1 \subset \mathbb{R}^{n_1}$, $E \subset
E_1 \times \mathbb{R}^{n_2}$, we check whether it is the case that

\begin{itemize}
\item[\refstepcounter{equation}\text{(\theequation)}\label{compute1}] {For
every $x \in E_1$, there exists $y \in \mathbb{R}^{n_2}$ such that $(x,y)
\in E$.}
\end{itemize}

If \eqref{compute1} holds, we compute a (possibly discontinuous) semialgebraic function $F: E_1
\rightarrow \mathbb{R}^{n_2}$ such that $(x, F(x)) \in E$ for all $x \in E_1$. \textup{(See Algorithm 11.3 as well as Section 5.1 in \cite{basu}.)}
\end{algorithm}

\begin{algorithm}\label%
{Algorithm-solution-to-module-ofPolynmialEq} Given a matrix of polynomials $%
\left[ A_{ij}\right] _{i=1,\cdots ,I;\text{ }j=1,\cdots ,J}$ with each $%
A_{ij}\in \mathbb{R}\left[ x_{1},\cdots ,x_{n}\right] $, we produce a list
of generators for the $\mathbb{R}\left[ x_{1},\cdots ,x_{n}\right] $-module
of all solutions $\vec{P}=\left( P_{1},\cdots ,P_{J}\right) $ (each $%
P_{j}\in \mathbb{R}\left[ x_{1},\cdots ,x_{n}\right] $) of the equations 
\begin{equation*}
\sum_{j=1}^{J}A_{ij}P_{j}=0\text{ (}i=1,\cdots ,I\text{).}
\end{equation*}%
Moreover, given $\left[ A_{ij}\right] $ as above, and given $\vec{Q}=\left(
Q_{1},\cdots ,Q_{I}\right) $ with each $Q_{i}\in \mathbb{R}\left[
x_{1},\cdots ,x_{n}\right] $, we decide whether the equations%
\begin{equation*}
\sum_{j=1}^{J}A_{ij}P_{j}=Q_{i}\text{ (}i=1,\cdots ,I\text{)}
\end{equation*}%
have a solution $\vec{P}=\left( P_{1},\cdots ,P_{J}\right) $ with $%
P_{1},\cdots ,P_{J}\in \mathbb{R}\left[ x_{1},\cdots ,x_{n}\right] $; if
there is a solution $\vec{P}$, then we produce one. {\textup{(See Sections 3.5-3.7 in \cite{grobner}.)}}
\end{algorithm}

\begin{algorithm}\label{Algorithm-Graph-Decomposition-Algorithm} (Graph Decomposition Algorithm)
Given a semialgebraic set $E\subset \mathbb{R}^{q}$, we compute a partition
of $E$ into finitely many semialgebraic sets $E_{\nu }$, for each of which
there is an invertible linear map $T_{\nu }:\mathbb{R}^{q}\rightarrow 
\mathbb{R}^{q}$ such that $T_{\nu }E_{\nu }$ has the form 
\begin{equation*}
T_{\nu }E_{\nu }=\left\{ 
\begin{array}{c}
\left( x_{1},\cdots ,x_{n_{\nu }},y_{1},\cdots ,y_{q-n_{\nu }}\right) \in 
\mathbb{R}^{q}: \\ 
\left( x_{1},\cdots ,x_{n_{\nu }}\right) \in U_{\nu },y_{\mu }=G_{\mu \nu
}\left( x_{1},\cdots ,x_{n_{\nu }}\right) \text{ } \\ 
\text{for }\mu =1,\cdots ,q-n_{\nu }%
\end{array}%
\right\} ,
\end{equation*}%
where $U_{\nu }\subset \mathbb{R}^{n_{\nu }}$ is a semialgebraic open set
and $G_{\mu \nu }:U_{\nu }\rightarrow \mathbb{R}$ is semialgebraic.
Moreover, we compute the above $T_{\nu },U_{\nu }$, and $G_{\mu \nu }$. {\textup{(See Algorithm 11.3 in \cite{basu}.)}}
\end{algorithm}

\subsection{Elimination of Quantifiers}

In this subsection, we discuss ``elimination of quantifiers", a powerful tool to show that
certain sets are semialgebraic, and to compute those sets.

The sets in question consist of all $(x_1, \cdots, x_n) \in \mathbb{R}^n$
that satisfy a certain condition $\Phi(x_1, \cdots, x_n)$. Here, $\Phi(x_1,
\cdots, x_n)$ is a statement in a formal language, the ``first order
predicate calculus for the theory of real closed fields".

Rather than giving careful definitions, we illustrate with a few examples,
and refer the reader to \cite{carlos,basu}.

\begin{itemize}
\item If $E \subset \mathbb{R}^{n_1} \times \mathbb{R}^{n_2}$ is a given
semialgebraic set, and if $\pi: \mathbb{R}^{n_1} \times \mathbb{R}^{n_2}
\rightarrow \mathbb{R}^{n_1}$ denotes the natural projection, then we can
compute the semialgebraic set $\pi E$, because $\pi E$ consists of all $%
(x_1, \cdots, x_{n_1}) \in \mathbb{R}^{n_1}$ satisfying the condition
\begin{equation*}
\Phi (x_1, \cdots, x_{n_1}): \left(\exists y_1 \right)\cdots \left(\exists
y_{n_2} \right) \left( (x_1, \cdots, x_{n_1}, y_1, \cdots, y_{n_2}) \in
E\right).
\end{equation*}

\item Suppose $E \subset \mathbb{R}^n$ is semialgebraic. Then we can compute 
$E^{\text{closure}}$, the closure of $E$, because $E^{\text{closure}}$
consists of all $(x_1, \cdots, x_n)$ satisfying the condition
\begin{equation*}
\Phi (x_{1},\cdots ,x_{n}):\left( \forall \varepsilon >0\right) \left(
\exists y_{1}\right) \cdots \left( \exists y_{n}\right) \left( \left[ \left(
y_{1},\cdots ,y_{n}\right) \in E\right] \wedge \left[ \left(
x_{1}-y_{1}\right) ^{2}+\cdots +\left( x_{n}-y_{n}\right) ^{2}<\varepsilon
^{2}\right] \right) \text{.}
\end{equation*}%
In particular, $E^{\text{closure}}$ is semialgebraic.

\item Let $E,\underline{E}\subset \mathbb{R}^{n}$ be given semialgebraic
sets. Then we can compute the semialgebraic set 
\begin{equation*}
\tilde{E}=\left\{ 
\begin{array}{c}
\left( x_{1},\cdots ,x_{n},\underline{x}_{1},\cdots ,\underline{x}%
_{n}\right) \in E\times \underline{E}:\left( \underline{x}_{1},\cdots ,%
\underline{x}_{n}\right) \text{ is at least as close as} \\ 
\text{any point of }\underline{E}\text{ to }\left( x_{1},\cdots
,x_{n}\right) 
\end{array}%
\right\} ,
\end{equation*}%
because $\tilde{E}$ consists of all $\left( x_{1},\cdots ,x_{n},\underline{x}%
_{1},\cdots ,\underline{x}_{n}\right) \in \mathbb{R}^{2n}$ satisfying the
condition 
\begin{multline*}
\Phi \left( x_{1},\cdots ,x_{n},\underline{x}_{1},\cdots ,\underline{x}%
_{n}\right) :\left[ \left( x_{1},\cdots ,x_{n}\right) \in E\right] \wedge %
\left[ \left( \underline{x}_{1},\cdots ,\underline{x}_{n}\right) \in 
\underline{E}\right] \wedge  \\
\left[ \left( \forall y_{1}\right) \cdots \left( \forall y_{n}\right)
\left\{ \left[ \left( y_{1},\cdots ,y_{n}\right) \in \underline{E}\right]
\rightarrow \left( x_{1}-\underline{x}_{1}\right) ^{2}+\cdots +\left( x_{n}-%
\underline{x}_{n}\right) ^{2}\leq \left( x_{1}-y_{1}\right) ^{2}+\cdots
+\left( x_{n}-y_{n}\right) ^{2}\right\} \right] \text{.}
\end{multline*}
\end{itemize}

\subsection{} We present the following trivial algorithm.
\begin{algorithm}\label{algorithmxxx}
Given a semialgebraic function $F: E \rightarrow \mathbb{R}$ defined on a semialgebraic set $E \subset \mathbb{R}^n$, we produce a nonzero polynomial $P(x,t)$ on $\mathbb{R}^n \times \mathbb{R}$ such that $P(x,F(x))=0$ for all $x \in E$.
\end{algorithm}

\begin{proof}[\textbf{Explanation}]
$F$ is specified to us by expressing its graph $\Gamma = \{(x,F(x)): x \in E\}$ as a Boolean combination of sets of the form $\{P > 0\}$ or $\{P = 0\}$ for nonzero polynomials $P$ on $\mathbb{R}^n \times \mathbb{R}$. Therefore, we can trivially express $\Gamma$ as a disjoint union over $\nu=1,\cdots, N$ of nonempty sets of the form 
\[
\Gamma_\nu = \left\{ (x,t) \in \mathbb{R}^n \times \mathbb{R}: \tilde{P}_{\nu i} (x,t) >0 \text{ for } i =1,\cdots, i_{\max}(\nu), P_{\nu j } (x,t) =0 \text{ for } j = 1, \cdots, j_{\max}(\nu) \right\}
\]
with each $i_{\max}(\nu), j_{\max}(\nu) \geq 0$. 
No $\Gamma_\nu$ can be open in $\mathbb{R}^n \times \mathbb{R}$, because ${\Gamma}_\nu \subset \Gamma$; hence $j_{\max}(\nu) \geq 1$ for each $\nu$. We take $P(x,t) = \prod_{\nu=1}^N P_{\nu 1} (x,t)$.
\end{proof}

\subsection{} Next, we present several refinements of Algorithm \ref{Algorithm-Graph-Decomposition-Algorithm}. We explain these algorithms in detail, although experts in computational algebraic geometry will find them routine.

To prepare the way, we present the following algorithms.

\begin{algorithm}
\label{algorithm6.1}Given a nonempty open semialgebraic subset $U\subset 
\mathbb{R}^{n}$ and a semialgebraic function $F:U\rightarrow \mathbb{R}$, we
compute semialgebraic sets $U_{junk}$, $U_{1},\cdots ,U_{N}\subset \mathbb{R}%
^{n}$, with the following properties

\begin{itemize}
\item $U$ is the disjoint union of $U_{junk}, U_1, \cdots, U_N$.

\item $U_{junk}$ has dimension strictly less than $n$.

\item Each $U_\nu$ ($\nu = 1, \cdots, N$) is open in $\mathbb{R}^n$.

\item $F|_{U_\nu}$ is continuous, for each $\nu=1,\cdots,N$.
\end{itemize}
\end{algorithm}

\begin{proof}[\textbf{Explanation}]
We start by recalling a useful property of roots of polynomials. Fix $D\geq
1 $. Given $a_{0},a_{1},\cdots ,a_{D-1}\in \mathbb{R}$, we write 
\begin{equation*}
\rt_{1}(a_{0},\cdots ,a_{D-1})\leq \rt_{2}(a_{0},\cdots ,a_{D-1})\leq \cdots %
\leq \rt_{D}(a_{0},\cdots ,a_{D-1})
\end{equation*}%
to denote the real parts of the roots of the polynomial 
\begin{equation*}
z^{D}+a_{D-1}z^{D-1}+\cdots +a_{0}=0,\text{ with multiplicities counted,}
\end{equation*}%
sorted in ascending order. Then we recall that the map $(a_{0},\cdots
,a_{D-1})\mapsto \rt_{k}(a_{0},\cdots ,a_{D-1})$ is continuous and
semialgebraic, for each fixed $k=1,\cdots ,D$.

Now let $U\subset \mathbb{R}^{n}$ be nonempty open and semialgebraic, and
let $F:U\rightarrow \mathbb{R}$ be semialgebraic. We can find a nonzero polynomial $%
P(x_{1},\cdots ,x_{n},z)$ such that%
\begin{equation*}
P\left( x_{1},\cdots ,x_{n},F\left( x_{1},\cdots ,x_{n}\right) \right) =0
\end{equation*}%
for all $\left( x_{1},\cdots ,x_{n}\right) \in U$. Note that $P\left(
x_{1},\cdots ,x_{n},z\right) $ cannot be independent of $z$, since $U$ is a
nonempty open subset of $\mathbb{R}^{n}$. Thus, we may write 
\begin{equation*}
P\left( x_{1},\cdots ,x_{n},z\right) =a\left( x\right)
z^{D}+\sum_{i=0}^{D-1}b_{i}\left( x\right) z^{i}
\end{equation*}%
for polynomials $a$, $b_{i}$ with $a$ nonzero. (We write $x$ to
denote $\left( x_{1},\cdots ,x_{n}\right) $.) We define $U_{junk,0}=\left\{
x\in U:a\left( x\right) =0\right\} $; thus, $U_{junk,0}\subset U$ is
semialgebraic and has dimension $\leq n-1$. On $U\setminus U_{junk,0}$, 
\begin{equation}
z=F\left( x\right) \text{ satisfies }z^{D}+\sum_{i=0}^{D-1}a_{i}\left(
x\right) z^{i}=0\text{, where }a_{i}\left( x\right) =\frac{b_{i}\left(
x\right) }{a\left( x\right) }\text{.}  \label{sec6-1}
\end{equation}%
Each $a_i \left( \cdot \right) $ is a continuous semialgebraic
(rational) function on $U\setminus U_{junk,0}$.

Thanks to \eqref{sec6-1}, we have $F(x)=\rt_k(a_0(x),a_1(x),%
\cdots,a_{D-1}(x))$ for some $k$ (each $x\in U\setminus U_{junk,0}$).

For each $k=1,\cdots ,D$, let%
\begin{equation*}
U_{1,k}=\left\{ 
\begin{array}{c}
x\in U\setminus U_{junk,0}:F\left( x\right) =\rt_{k}\left( a_{0}\left(
x\right) ,a_{1}\left( x\right) ,\cdots ,a_{D-1}\left( x\right) \right) ,%
\text{ but} \\ 
F\left( x\right) \not=\rt_{k^{\prime }}\left( a_{0}\left( x\right)
,a_{1}\left( x\right) ,\cdots ,a_{D-1}\left( x\right) \right) \text{ for all 
}k^{\prime }<k%
\end{array}%
\right\} .
\end{equation*}%
Thus $U$ is the disjoint union of $U_{junk,0}$ and the $U_{1,k}$ ($%
k=1,\cdots ,D$). Moreover, each $U_{1,k}$ is a semialgebraic subset of $%
\mathbb{R}^{n}$, and $F\left( x\right) =\rt_{k}\left( a_{0}\left( x\right)
,a_{1}\left( x\right) ,\cdots ,a_{D-1}\left( x\right) \right) $ on $U_{1,k}$%
. Since the $a_{i}\left( x\right) $ are continuous on $U\setminus U_{junk,0}$%
, and since $\left( a_{0},\cdots ,a_{D-1}\right) \mapsto \rt_{k}\left(
a_{0},\cdots ,a_{D-1}\right) $ is continuous, it follows that $F|_{U_{1,k}}$
is continuous.

We now compute semialgebraic sets $U_{k,junk},U_{k}$ with the following
properties (see Section \ref{section-known-algorithms})

\begin{itemize}
\item $U_{1,k}$ is the disjoint union of $U_k$ and $U_{k,junk}$.

\item $U_{k,junk}$ has dimension $\leq n-1$.

\item $U_k$ is open in $\mathbb{R}^n$.
\end{itemize}
We define $U_{junk}=U_{junk,0}\cup \bigcup\limits_{k=1}^{D}U_{k,junk}$.

Thus, $U_{junk}$ is semialgebraic and has dimension $\leq n-1$; moreover, $U$
is the disjoint union of the $U_{junk}$, $U_{1},\cdots ,U_{D}$. Also, $%
U_{k}\subset \mathbb{R}^{n}$ is open, and $F|_{U_{k}}$ is continuous, for
each $k=1,\cdots ,D$. This concludes our explanation of Algorithm \ref%
{algorithm6.1}. \end{proof}

\begin{algorithm}
\label{algorithm6.2} Given a nonempty semialgebraic open set $U\subset 
\mathbb{R}^n$, and given a semialgebraic function $F: U \rightarrow \mathbb{R%
}$, we compute semialgebraic subsets $U_{junk},U_1,\cdots,U_N \subset U$
with the following properties:

\begin{itemize}
\item $U$ is the disjoint union of $U_{junk}, U_1,\cdots,U_N$.

\item $U_{junk}$ has dimension $\leq n-1$.

\item For each $\nu$ ($1\leq \nu \leq N$), $U_\nu$ is a nonempty open subset
of $\mathbb{R}^n$ and $F|_{U_\nu}$ is real-analytic.
\end{itemize}
\end{algorithm}

\begin{proof}[\textbf{Explanation}]
Using Algorithm \ref{algorithm6.1} we may easily reduce matters to the case
in which $F$ is continuous on $U$. As in the explanation of Algorithm \ref%
{algorithm6.1}, we can produce a polynomial%
\begin{equation*}
P\left( x,z\right) =a\left( x\right) z^{D}+\sum_{i=0}^{D-1}b_{i}\left(
x\right) z^{i}\text{ on }\mathbb{R}^{n}\times \mathbb{R},
\end{equation*}%
with $a\left( x\right) $ nonzero, such that 
\begin{equation*}
P\left( x,F\left( x\right) \right) =0\text{ for all }x\in U\text{.}
\end{equation*}%
We partition $U$ into the following semialgebraic sets: 
\begin{equation*}
\hat{U}_{l}=\left\{ x\in U:\left( \partial _{z}\right) ^{l^{\prime
}}P|_{\left( x,F\left( x\right) \right) }=0\text{ for all }l^{\prime }\leq l,%
\text{ but not for }l^{\prime }=l+1\right\}
\end{equation*}%
for $l=0,1,\cdots ,D-1$;%
\begin{equation*}
\hat{U}_{final}=\left\{ x\in U:\left( \partial _{z}\right) ^{l^{\prime
}}P|_{\left( x,F\left( x\right) \right) }=0\text{ for all }l^{\prime
}=0,1,\cdots ,D\right\} \text{.}
\end{equation*}%
Taking $l^{\prime }=D$ in the definition of $\hat{U}_{final}$, we see that 
\begin{equation*}
\hat{U}_{final}\subset \left\{ x\in U:a\left( x\right) =0\right\} ,
\end{equation*}%
and therefore $\hat{U}_{final}$ has dimension $\leq n-1$. For fixed $l$ ($%
0\leq l\leq D-1$), let $Q_{l}(x,\xi )=\left( \partial _{\xi }\right)
^{l}P\left( x,\xi \right) $. Then 
\begin{equation}
\text{for any }x\in \hat{U}_{l},\text{ we have }Q_{l}\left( x,F\left(
x\right) \right) =0\text{ but }\partial _{\xi }Q_{l}|_{\left( x,F\left(
x\right) \right) }\not=0\text{.}  \label{sec6-2}
\end{equation}%
Applying an algorithm from Section \ref{section-known-algorithms}, we
partition $\hat{U}_{l}$ into a semialgebraic set $U_{l,junk}$ of dimension $%
\leq n-1$, and a semialgebraic open set $U_{l}$ (possibly empty). Thus, our
original set $U$ is partitioned into the semialgebraic sets, $U_{l}$ ($l=0,1,\cdots ,D-1$), and
\begin{equation*}
U_{junk}:=\hat{U}_{final}\cup \bigcup_{l=1}^{D-1}U_{l,junk}\text{,}
\end{equation*}%
with $\dim U_{junk}\leq n-1$; the $U_{l}$
are open.

We now show that $F|_{U_{l}}$ is real-analytic for each $l=0,\cdots ,D-1$.
Fix $x^{0}\in U_{l}\subset \hat{U}_{l}$, and let $\xi ^{0}=F(x^{0})$. By (%
\ref{sec6-2}), we have 
$Q_{l}\left( x^{0},\xi ^{0}\right) =0 $
but $\partial _{\xi}Q_{l}\left( x^{0},\xi ^{0}\right) \not=0$. Hence, thanks
to the real analytic implicit function theorem, we know that for small enough $%
\varepsilon >0$, the set 
\begin{equation*}
\left\{ \left( x,\xi \right) \in \mathbb{R}^{n}\times \mathbb{R}:\left\vert
x-x^{0}\right\vert ,\left\vert \xi -\xi ^{0}\right\vert <\varepsilon
,Q_{l}\left( x,\xi \right) =0\right\}
\end{equation*}%
is contained in the graph of a real-analytic function $z=\psi \left(
x\right) $ defined on some neighborhood of $x^{0}$.

For $x\in U_{l}\subset \hat{U}_{l}$, \eqref{sec6-2} tells us that $%
Q_{l}(x,F(x))=0$. Moreover, $F$ is continuous on $U$, hence on $U_{l}$.
Therefore, for small enough $\delta >0$, we have 
\begin{eqnarray*}
&&\left\{ \left( x,F\left( x\right) \right) :\left\vert x-x^{0}\right\vert
<\delta \right\} \\
&\subset &\left\{ \left( x,\xi \right) \in \mathbb{R}^{n}\times \mathbb{R}%
:\left\vert x-x^{0}\right\vert ,\left\vert \xi -\xi ^{0}\right\vert
<\varepsilon ,Q_{l}\left( x,\xi \right) =0\right\} .
\end{eqnarray*}%
Therefore, $F$ agrees with $\psi $ on the ball $B(x^{0},\delta )$. This
proves that $F$ is real-analytic in a neighborhood of $x^{0}$. Since $%
x^{0}\in U_{l}$ was picked arbitrarily, we now know that $F$ is
real-analytic on $U_{l}$.

We now delete any empty $U_l$ from our list of open sets $U_{0},\cdots ,U_{D-1}$. The remaining list, together with the set $U_{junk}$ defined and computed above have all the properties asserted in Algorithm \ref%
{algorithm6.2}. The explanation of that algorithm is complete.
\end{proof}

\begin{algorithm}
\label{algorithm6.3}Given a nonempty semialgebraic set $U\subseteq \mathbb{R}%
^{n}$, and given semialgebraic functions $F_{1},\cdots ,F_{K}:U\rightarrow 
\mathbb{R}$, we compute a partition of $U$ into semialgebraic subsets $%
U_{junk},U_{1},\cdots ,U_{N}$ with the following properties

\begin{itemize}
\item $U_{junk}$ has dimension $\leq n-1$.

\item Each $U_\nu$ is a nonempty open subset of $\mathbb{R}^n$.

\item The restriction of each $F_k$ to each $U_\nu$ is real-analytic.
\end{itemize}
\end{algorithm}

\begin{proof}[\textbf{Explanation}]
We proceed by recursion on $K$, using Algorithm \ref{algorithm6.2}. If $K=1$%
, Algorithm \ref{algorithm6.2} does the job. If $K>1$, then by applying
Algorithm \ref{algorithm6.2}, we partition $U$ into $\hat{U}_{junk}$ and $%
\hat{U}_{i}$ ($i=1,\cdots ,I$) (semialgebraic subsets), with the following
properties: $\dim \hat{U}_{junk}\leq n-1$; each $\hat{U}_{i}$ is non-empty
and open; $F_{K}|_{\hat{U}_{i}}$ is real-analytic. We recursively apply
Algorithm \ref{algorithm6.3}, with $K$ replaced by $K-1$, to the non-empty
open semialgebraic set $\hat{U}_{i}$ and the semialgebraic functions $%
F_{1}|_{\hat{U}_{i}},\cdots ,F_{K-1}|_{\hat{U}_{i}}$. Thus, we partition $%
\hat{U}_{i} $ into semialgebraic sets $\hat{U}_{i,junk}$ and $\hat{U}_{i,\nu
}$ ($\nu =1,\cdots ,\nu _{\max }(i))$, with the following properties: $\hat{U%
}_{i,junk}$ has dimension $\leq n-1$; each $\hat{U}_{i,\nu }$ is non-empty
and open; $F_{k}|_{\hat{U}_{i,\nu }}$ is real-analytic for $k=1,\cdots ,K-1$.

In fact, $F_{k}|_{\hat{U}_{i,\nu }}$ is real-analytic for $k=1,\cdots ,K$,
since $F_{K}|_{\hat{U}_{i,\nu }}$ is real-analytic.

We now define $U_{junk}=\hat{U}_{junk}\cup \bigcup_{i=1}^{I}\hat{U}%
_{i,junk}, $ and let $U_{1},\cdots ,U_{N}$ be an enumeration of the $\hat{U}%
_{i,\nu }$ ($i=1,\cdots ,I,\nu =1,\cdots ,\nu _{\max }\left( i\right) $).

These sets are semialgebraic; $\dim U_{junk}\leq n-1$; each $U_{\nu }$ is
nonempty and open; and $F_{k}|_{U_{\nu }}$ is real-analytic ($k=1,\cdots ,K$%
) for each $\nu $. This concludes the explanation of Algorithm \ref%
{algorithm6.3}.
\end{proof}

Now we can present our refinements of Algorithm \ref{Algorithm-Graph-Decomposition-Algorithm}.

\begin{algorithm}
\label{algorithm-6.4}Given a semialgebraic set $E\subset \mathbb{R}^{q}$ of
dimension $\leq n$, and given semialgebraic functions $F_{1},\cdots
,F_{K}:E\rightarrow \mathbb{R}$, we compute a decomposition of $E$ into
semialgebraic sets $E_{junk}$, $E_{1},\cdots ,E_{N}$ with the following
properties:

\begin{itemize}
\item $\dim E_{junk} \leq n-1$

\item For each $\nu = 1, \cdots, N$ there exists an invertible linear map $%
T_\nu: \mathbb{R}^q \rightarrow \mathbb{R}^q$ such that $T_\nu E_\nu =
\{(x,y) \in \mathbb{R}^n \times \mathbb{R}^{q-n}: x \in U_\nu, y = G_\nu(x)
\}$, where $U_\nu \subset \mathbb{R}^{q-n}$ is a semialgebraic nonempty open
set and $G_\nu$ is a semialgebraic map.

\item For each $\nu = 1, \cdots, N$ and each $\lambda = 1, \cdots, K$, we
have $F_\lambda \circ T_\nu^{-1} (x,y) = H_{\lambda \nu} (x)$ for all $x \in
U_\nu$, $y = G_\nu (x)$ (i.e., for all $(x,y) \in T_\nu E_\nu$), where $%
H_{\lambda \nu}: U_\nu \rightarrow \mathbb{R}$ is a semialgebraic function.

\item Moreover, we compute $T_\nu, G_\nu, H_{\lambda \nu}$ as above. Also,
for each $\nu = 1, \cdots,N$ and each $\mu= 1,\cdots, q-n_\nu$, we compute a
nonzero polynomial $P_{\mu \nu} (x_1,\cdots, x_n, y_\mu)$ such that $%
P_{\mu \nu} (x_1, \cdots, x_n,y_\mu)=0$ for all $(x_1,\cdots, x_n,
y_1,\cdots, y_{q-n}) \in T_{\nu} E_\nu$.

\item Finally, for each $\nu =1,\cdots ,N$ and each $\lambda =1,\cdots ,K$,
we compute a nonzero polynomial $\hat{P}_{\lambda \nu }(x_{1},\cdots
,x_{n},z_{\lambda })$ such that $\hat{P}_{\lambda \nu }(x_{1},\cdots
,x_{n},z_{\lambda })=0$ for $(x_{1},\cdots ,x_{n},y_{1},\cdots ,y_{q-n})\in
T_{\nu }E_{\nu }$ and $z_{\lambda }=F_{\lambda }\circ T_{\nu
}^{-1}(x_{1},\cdots ,x_{n},y_{1},\cdots ,y_{q-n})$.
\end{itemize}
\end{algorithm}

\begin{proof}[\textbf{Explanation}]
We first apply Algorithm \ref{Algorithm-Graph-Decomposition-Algorithm} to $E$%
, and throw away all the $E_\nu$ with $\dim E_\nu <n $ into $E_{junk}$.
Thus, $E_{junk} $ is a semialgebraic set of dimension $\leq n-1$. There
remain the $E_\nu$ with $\dim E_\nu=n$. For each such an $E_\nu$, Algorithm %
\ref{Algorithm-Graph-Decomposition-Algorithm} provides $T_\nu, U_\nu, G_\nu$
as described in Algorithm \ref{algorithm-6.4}. Moreover, once we know $U_\nu$
and $G_\nu$, we compute a nonzero polynomial $P_{\mu
\nu}(x_1,\cdots, x_n,y_\mu)$ as in the statement of Algorithm \ref%
{algorithm-6.4}.

Next, since $T_\nu E_\nu = \{(x,y): x \in U_\nu, y = G_\nu (x) \}$ with $%
U_\nu, G_\nu$ semialgebraic, we can compute semialgebraic functions $%
H_{\lambda \nu}: U_{\nu} \rightarrow \mathbb{R}$ as in the statement of
Algorithm \ref{algorithm-6.4}. Once we know $U_\nu$ and $H_{\lambda \nu}$,
we can easily compute nonzero polynomials $\hat{P}_{\lambda \nu}$ as in the
statement of Algorithm \ref{algorithm-6.4}. This completes the explanation
of Algorithm \ref{algorithm-6.4}.
\end{proof}

\begin{algorithm}
\label{algorithm-6.4'} Given a semialgebraic set $E \subset \mathbb{R}^q$ of
dimension $\leq n$, and given semialgebraic functions $F_1,\cdots, F_K: E
\rightarrow \mathbb{R}$, we compute objects $E_{junk}, E_1,\cdots, E_N,
T_\nu, U_\nu, G_\nu, H_{\lambda \nu}, P_{\mu \nu}, \hat{P}_{\mu \nu}$ having
all the properties asserted in Algorithm \ref{algorithm-6.4}, but also
satisfying the following:

\begin{itemize}
\item Each $U_\nu$ is connected.

\item Each $G_\nu, H_{\lambda \nu}$ is real analytic (on $U_{\nu}$).
\end{itemize}
\end{algorithm}

\begin{proof}[\textbf{Explanation}]
First, we execute Algorithm \ref{algorithm-6.4}. For each $\nu $, we write $%
G_{\mu \nu }(x)$ to denote the $\mu ^{\text{th}}$ component of $G_\nu (x)\in 
\mathbb{R}^{q-n}$. We apply Algorithm \ref{algorithm6.3} to the open set $%
U_{\nu }$ and the list of functions consisting of the $G_{\mu \nu }\quad
(\mu =1,\cdots ,q-n)$ and the $H_{\lambda \nu }\quad (\lambda =1,\cdots ,K)$.

Thus, $U_{\nu }$ is partitioned into semialgebraic open sets $U_{\nu i}\quad
(i=1,\cdots ,I(\nu ))$ and a set $U_{\nu ,junk}$ of dimension $\leq n-1$; on
each $U_{\nu i}$, all the $G_{\mu ,\nu }$ and $H_{\lambda \nu }$ are
real-analytic. Using Algorithm \ref{Algorithm-computing-connected-component}
we further partition $U_{\nu i}$ into its connected components $U_{\nu ij}$.%
\newline
We now define: 
\begin{eqnarray*}
E_{junk}^{\ast } &=&E_{junk}\cup \bigcup_{\nu }T_{\nu }^{-1}\left\{ \left(
x,y\right) :x\in U_{\nu ,junk},y=G_{\nu }\left( x\right) \right\} \\
&&%
\begin{array}{ll}
T_{\nu ij}^{\ast }=T_{\nu } & E_{\nu ij}^{\ast }=T_{\nu }^{-1}\left\{ \left(
x,y\right) :x\in U_{\nu ij},y=G_{\nu }\left( x\right) \right\} \\ 
U_{\nu ij}^{\ast }=U_{\nu ij} & P_{\mu \nu ij}^{\ast }=P_{\mu \nu } \\ 
G_{\nu ij}^{\ast }=G_{\nu }|_{U_{\nu ij}} & \hat{P}_{\lambda \nu ij}^*=\hat{P}%
_{\lambda \nu } \\ 
H_{\lambda \nu ij}^{\ast }=H_{\lambda \nu }|_{U_{\nu ij}} & 
\end{array}%
\text{.}
\end{eqnarray*}%
One checks easily that $E_{junk}^{\ast }$ is semialgebraic and has dimension 
$\leq n-1$; $E$ is partitioned into $E_{junk}^{\ast }$ and the $E_{\nu
ij}^{\ast }$; each $E_{\nu ij}^{\ast }$ is semialgebraic; $T_{\nu ij}^{\ast
}E_{\nu ij}^{\ast }=\left\{ \left( x,y\right) :x\in U_{\nu ij}^{\ast
},y=G_{\nu ij}^{\ast }\left( x\right) \right\} $; $U_{\nu ij}^{\ast }\subset 
\mathbb{R}^{n}$ is nonempty, open, connected and semialgebraic; $G_{\nu ij}^{\ast
}:U_{\nu ij}^{\ast }\rightarrow \mathbb{R}^q$ is a real-analytic semialgebraic map; $H^{\ast}_{\lambda \nu ij}: U^{\ast}_{\nu ij} \rightarrow \mathbb{R}$ is semialgebraic and real-analytic; 
$F_{\lambda }\circ \left( T_{\nu ij}^{\ast }\right) ^{-1}\left( x,y\right)
=H_{\lambda \nu ij}^{\ast }\left( x\right) $ for $x\in U_{\nu ij}^{\ast }$, 
$y=G_{\nu ij}^{\ast }\left( x\right) $; $P_{\mu ,\nu ij}^{\ast }\left(
x,y_{\mu }\right) =0$ for $\left( x,y_{1},\cdots ,y_{q-n}\right) \in T_{\nu
_{ij}}^{\ast }E_{\nu ij}^{\ast }$; and $\hat{P}_{\lambda \nu ij}^{\ast }\left(
x,z_{\lambda }\right) =0$ for $\left( x,y\right) \in T_{\nu ij}^{\ast
}E_{\nu ij}^{\ast }$ and $z_{\lambda }=F_{\lambda }\circ \left( T_{\nu
}^{\ast }\right) ^{-1}\left( x,y\right) $. \newline
The above objects have all the properties asserted in Algorithm \ref%
{algorithm-6.4'}. This completes the explanation of Algorithm \ref%
{algorithm-6.4'}.
\end{proof}

Our main refinement of Algorithm \ref{Algorithm-Graph-Decomposition-Algorithm} is the following.

\begin{algorithm}
\label{algorithm-6.5}Given a semialgebraic set $E\subset \mathbb{R}^{q}$,
and given semialgebraic functions $F_{1},\cdots ,F_{K}:E\rightarrow \mathbb{R%
}$, we compute a partition $E_{1},\cdots ,E_{N}$ of $E$, satisfying the following for each $\nu$:

\begin{itemize}
\item $E_{\nu }$ has the form 
\begin{equation*}
E_{\nu }=T_{\nu }^{-1}\left\{ \left( x,y\right) \in \mathbb{R}^{n_{\nu
}}\times \mathbb{R}^{q-n_{\nu }}:x\in U_{\nu },y=G_{\nu }\left( x\right)
\right\} \text{,}
\end{equation*}%
where $T_{\nu }:\mathbb{R}^{q}\rightarrow \mathbb{R}^{q}$ is an invertible
linear map, $U_{\nu }\subset \mathbb{R}^{n_{\nu }}$ is a connected, open
semialgebraic set, and $G_{\nu }:U_{\nu }\rightarrow \mathbb{R}^{q-n_{\nu }}$
is real-analytic and semialgebraic.

\item For each $\lambda =1,\cdots ,K$, we have 
\begin{equation*}
F_{\lambda }\circ T_{\nu }^{-1}\left( x,G_{\nu }\left( x\right) \right)
=H_{\lambda \nu }\left( x\right) \text{ for all }x\in U_{\nu }\text{,}
\end{equation*}%
where $H_{\lambda \nu }:U_{\nu }\rightarrow \mathbb{R}$ is real-analytic and
semialgebraic.
\end{itemize}

Moreover, we compute the above $T_{\nu },n_{\nu },U_{\nu },G_{\nu
},H_{\lambda \nu }$. In addition, we compute nonzero polynomials $P_{\mu \nu
}\left( x,y_{\nu }\right) $ $\left( \mu =1,\cdots ,q-n_{\nu }\right) $ and $%
\hat{P}_{\lambda \nu }\left( x,z_{\lambda }\right) $ $\left( \lambda
=1,\cdots ,K\right) $ such that

\begin{itemize}
\item $P_{\mu \nu }\left( x,y_{\mu }\right) =0$ for $x\in U_{\nu }$, $\left(
y_{1},\cdots ,y_{q-n_{\nu }}\right) =G_{\nu }\left( x\right) $; and $\hat{P}%
_{\lambda \nu }\left( x,H_{\lambda \nu }\left( x\right) \right) =0$ for $%
x\in U_{\nu }$.
\end{itemize}
\end{algorithm}

\begin{proof}[\textbf{Explanation}]
We proceed by recursion on $\dim E$. If $\dim E=0$, then $E$ consists of
finitely many points, and our task is trivial.\newline
Fix $n\geq 1$, and suppose we can carry out Algorithm \ref{algorithm-6.5}
whenever $\dim E\leq n-1$. We explain how to carry out the algorithm when $%
\dim E=n$. \newline
To do so, we apply Algorithm \ref{algorithm-6.4'}. The semialgebraic sets $%
E_{\nu }$ have the form desired for Algorithm \ref{algorithm-6.5}, and we
obtain also the corresponding $T_{\nu },U_{\nu },G_{\nu },P_{\mu \nu },\hat{P%
}_{\mu \nu }$ (with $n_{\nu }=n$). \newline
However, we still have to deal with the semialgebraic set $E_{junk}$ arising
from Algorithm \ref{algorithm-6.4'}. Since $\dim E_{junk}\leq n-1$, we may
recursively apply Algorithm \ref{algorithm-6.5} to the set $E_{junk}$ and
the functions $F_{1}|_{E_{junk}},\cdots ,F_{K}|_{E_{junk}}$.\newline
It is now trivial to carry out Algorithm \ref{algorithm-6.5} for the given $%
E,F_{1},\cdots ,F_{K}$.
\end{proof}

\subsection{} In this subsection, we explain how to compute generators for the ideal of polynomials that vanish on given semialgebraic set.

The strategy is to apply Algorithm \ref{algorithm-6.5}, and then execute the following procedure.

\begin{algorithm}[Complexify]\label{algorithm-complexify}
Let $\Gamma =\left\{ \left( x,G\left( x\right) \right) :x\in U\right\} $
where $U\subset \mathbb{R}^{n}$ is an open, connected, nonempty
semialgebraic set, and $G\left( x\right) =\left( G_{1}\left( x\right)
,\cdots ,G_{m}\left( x\right) \right) $ is a real-analytic semialgebraic map
from $U$ to $\mathbb{R}^{m}$. Suppose $P_{1}\left( x,t\right) ,\cdots
,P_{m}\left( x,t\right) $ are nonzero (real) polynomials on $\mathbb{R}%
^{n}\times \mathbb{R}$ such that $P_{\mu }\left( x,G_{\mu }\left( x\right)
\right) =0$ on $U$ for all $\mu =1,\cdots ,m$. Given $U,G,P_{1},\cdots ,P_{m}
$ as above, we compute an affine variety $V\subset \mathbb{C}^{n}\times 
\mathbb{C}^{m}$ with the following property: Let $P\in \mathbb{C}\left[
z_{1},\cdots ,z_{n},w_{1},\cdots ,w_{m}\right] $. Then $P$ vanishes
everywhere on $\Gamma $ if and only if $P$ vanishes everywhere on $V$. 

\begin{proof}[Explanation]
We first reduce to the case in which, for each $\mu =1,\cdots ,m$, $\partial
_{t}P_{\mu }\left( x,t\right) $ isn't identically zero on the graph of $%
G_{\mu }:U\rightarrow \mathbb{R}$. To do so, note that if for some $\mu $,
we have 
\begin{equation}
\partial _{t}P_{\mu }\left( x,G_{\mu }\left( x\right) \right) =0\text{ for
all }x\in U,  \label{x1}
\end{equation}%
then we may simply replace $P_{\mu }$ by $\partial _{t}P_{\mu }$, preserving
all the assumptions made in Algorithm \ref{algorithm-complexify}. Note
that we can decide whether (\ref{x1}) holds, and note that $\partial
_{t}P_{\mu }$ is a nonzero polynomial. (Otherwise we would have $P_{\mu
}\left( x,t\right) =P_{\mu }^{\#}\left( x\right) $ for a nonzero polynomial $%
P_{\mu }^{\#}$; our assumption $P_{\mu }\left( x,G_{\mu }\left( x\right)
\right) =0$ would then imply that $P_{\mu }^{\#}\left( x\right) =0$ for all $%
x$ in the nonempty open set $U$, which is impossible.) Each time we replace $%
P_{\mu }$ by $\partial _{t}P_{\mu }$ for some $\mu $, the quantity $%
\sum_{\mu =1}^{m}\deg P_{\mu }$ decreases. Therefore, after finitely many
steps, we arrive at the case in which, for each $\mu =1,\cdots ,m$, $%
\partial _{t}P_{\mu }$ isn't identically zero on the graph of $G_{\mu }$.

From now on, we assume we are in this case. The functions $U\ni x\mapsto
\partial _{t}P_{\mu }\left( x,G_{\mu }\left( x\right) \right) $ $\left( \mu
=1,\cdots ,m\right) $ are real-analytic and nonzero. Hence,
there exists $x^{0}\in U$ such that 
\begin{equation}
\partial _{t}P_{\mu }\left( x^{0},G_{\mu }\left( x^{0}\right) \right) \not=0%
\text{ for each }\mu =1,\cdots ,m.  \label{xx1}
\end{equation}%
We can compute such an $x^{0}\in U$, by standard algorithms in Section \ref{section-known-algorithms}. Let $y^{0}=\left( y_{1}^{0},\cdots ,y_{m}^{0}\right)
=G\left( x^{0}\right) $. Thus, $\left( x^{0},y^{0}\right) \in \Gamma $, and $%
\left( x^{0},y^{0}\right) $ belongs to the algebraic set 
\begin{equation}
W=\left\{ \left( z,w_{1},\cdots ,w_{m}\right) \in \mathbb{C}^{n}\times 
\mathbb{C}^{m}:P_{\mu }\left( z,w_{\mu }\right) =0\text{ for }\mu =1,\cdots
,m\right\} .  \label{xx2}
\end{equation}%
Moreover, thanks to (\ref{xx1}), we know that 

\LAQ{xx3}{The differentials $%
dP_{\mu }\left( x^{0},y^{0}\right) $ $\left( \mu =1,\cdots ,m\right) $ are
linearly independent, where we regard each $P_{\mu }\left( x,w_{\mu }\right) 
$ as a polynomial in $\left( x,w_{1},\cdots ,w_{n}\right) $.} Hence we may
apply the holomorphic implicit function theorem and Lemma \ref{lemma5} in Section \ref{preliminary}. From the holomorphic implicit function theorem we obtain a ball $%
B_{1}\subset \mathbb{C}^{n}$ centered at $x^{0}$, an open ball $B_{2}\subset 
\mathbb{C}^{n}\times \mathbb{C}^{m}$ centered at $\left( x^{0},y^{0}\right) $%
, and a holomorphic map $G^{\mathbb{C}}:B_{1}\rightarrow \mathbb{C}^{m}$
such that $G^{\mathbb{C}}\left( x^{0}\right) =y^{0}$, and 
\begin{equation}
W\cap B_{2}=\left\{ \left( z,G^{\mathbb{C}}\left( z\right) \right) :z\in
B_{1}\right\} \cap B_{2}\text{.}  \label{xx4}
\end{equation}%
We may suppose $B_{1}\cap \mathbb{R}^{n}\subset U$. Note that $\Gamma
\subset W$, hence (\ref{xx4}) gives $G\left( x\right) =G^{\mathbb{C}}\left(
x\right) $, provided $x\in B_{1}\cap \mathbb{R}^{n}$, and 
\begin{equation}
\left( x,G\left( x\right) \right) ,\left( x,G^{\mathbb{C}}\left( x\right)
\right) \in B_{2}.  \label{xx5}
\end{equation}%
Moreover, both $G$ and $G^{\mathbb{C}}$ are continuous, and we have $\left(
x^{0},G\left( x^{0}\right) \right) =\left( x^{0},G^{\mathbb{C}}\left(
x^{0}\right) \right) =\left( x^{0},y^{0}\right) \in B_{2}$. Therefore, (\ref%
{xx5}) holds for all $x\in U$ sufficiently close to $x^{0}$. Consequently, 
\begin{equation}
G\left( x\right) =G^{\mathbb{C}}\left( x\right) \text{ for all }x\in U\text{
sufficiently close to }x^{0}\text{.}  \label{xx6}
\end{equation}%

Returning to (\ref{xx2}) and (\ref{xx3}), we now apply Lemma \ref{lemma5} in Section \ref{preliminary}. Let $V_{1},\cdots ,V_{p}$ be the irreducible components
of $W$. (We can compute them, see Algorithm \ref{Algorithm-computing-connected-component}.) Let $V_{1,\text{reg}},\cdots
,V_{p,\text{reg}}$ denote the set of regular points of $V_{1},\cdots ,V_{p}$%
, respectively. Then Lemma \ref{lemma5} produces an index $%
j_{0}$ and a small ball $\tilde{B}\subset \mathbb{C}^{n}\times \mathbb{C}^{m}
$ about $\left( x^{0},y^{0}\right) $, for which \ the following hold. 
\begin{equation}
W\cap \tilde{B}=V_{j_{0},\text{reg}}\cap \tilde{B}\text{, and }V_{j}\cap 
\tilde{B}=\emptyset \text{ for }j\not=j_{0}.  \label{xx7}
\end{equation}%
In particular, $j_{0}$ is the one and only $j\in \left\{ 1,\cdots ,p\right\} 
$ such that $\left( x^{0},y^{0}\right) \in V_{j}$. Hence, we can compute $%
j_{0}$. Now let $P\left( z,w\right) $ be any polynomial on $\mathbb{C}%
^{n}\times \mathbb{C}^{m}$. We claim that 
\LAQ{xx8}{$P$ vanishes everywhere on $%
\Gamma $ if and only if $P$ vanishes everywhere on $V_{j_{0}}$.}
If (\ref{xx8}) holds, then we can take $V$ in Algorithm \ref{algorithm-complexify} to be $V_{j_{0}}$, and we are done. Thus, the explanation of
Algorithm \ref{algorithm-complexify} is reduced to the task of proving
(\ref{xx8}). 

Suppose $P=0$ everywhere on $\Gamma $. Then $P\left( x,G\left(
x\right) \right) =0$ for all $x$ in a small neighborhood of $x^{0}$ in $%
\mathbb{R}^{n}$. Because $G^{\mathbb{C}}$ is holomorphic in a small
neighborhood of $x^{0}$ in $\mathbb{C}^{n}$, it follows from (\ref{xx6})
that $P\left( z,G^{\mathbb{C}}\left( z\right) \right) =0$ for all $z$ in a
small neighborhood of $x^{0}\in \mathbb{C}^{n}$. Therefore, by (\ref{xx4}),
we have \LAQ{xx9}{$P\left( z,w\right)=0 $ for all $\left( z,w\right) \in W$ sufficiently
close to $\left( x^{0},y^{0}\right) $.} Thanks to (\ref{xx7}) and (\ref{xx9}%
), we have $\left( x^{0},y^{0}\right) \in V_{j_{0},\text{reg}}$, and $%
P\left( z,w\right) =0$ for all $\left( z,w\right) \in V_{j_{0},\text{reg}}$
sufficiently close to $\left( x^{0},y^{0}\right) $. Because $P$ is a
polynomial and $V_{j_{0},\text{reg}}$ is a connected complex-analytic submanifold of $\mathbb{%
C}^{n}$, it follows that $P\left( z,w\right) =0$ for all $\left( z,w\right)
\in V_{j_{0},\text{reg}}$. Because $P$ is continuous and $V_{j_{0},\text{reg}%
}$ is dense in $V_{j_{0}}$, it follows that $P\left( z,w\right) =0$ for all $%
\left( z,w\right) \in V_{j_{0}}$. Thus, we have shown that if $P=0$
everywhere on $\Gamma $, then $P=0$ everywhere on $V_{j_{0}}$. 

Conversely,
suppose $P=0$ everywhere on $V_{j_{0}}$. Then, in particular, $P=0$
everywhere on $V_{j_{0},\text{reg}}$. Applying (\ref{xx7}), we have that $P=0
$ on a small neighborhood of $\left( x^{0},y^{0}\right) $ in $W$. Thanks to (%
\ref{xx4}) and the continuity of $G^{\mathbb{C}}$, this yields $P\left( z,G^{\mathbb{C}}\left( z\right) \right) =0$
for all $z\in \mathbb{C}^{n}$ sufficiently close to $x^{0}$. Consequently,
by (\ref{xx6}), we have $P\left( x,G\left( x\right) \right) =0$ for all $%
x\in U$ sufficiently close to $x^{0}$. However, we assume in Algorithm \ref{algorithm-complexify} that $G$ is real-analytic and $U$ is
connected. Therefore, $P\left( x,G\left( x\right) \right) =0$ for all $x$ $%
\in U$ . Thus, $P=0$ everywhere on $\Gamma $. This completes the proof of (%
\ref{xx8}) and the explanation of Algorithm \ref{algorithm-complexify}. 
\end{proof}
\end{algorithm}

\begin{algorithm}\label{Algorithm-generators-for-ideal-vanishing-on-semiA}
Given a semialgebraic set $E\subset \mathbb{R}^{q}$, we compute generators
for the ideal of all polynomials $P\in \mathbb{R}\left[ x_{1},\cdots ,x_{q}%
\right] $ that vanish on $E$.
\end{algorithm}

\begin{proof}[Explanation]
Applying Algorithm \ref{algorithm-6.5} (with, say, a single $%
F_{\lambda }\equiv 1$), we obtain a partition of $E$ into finitely many
semialgebraic subsets $E_{\nu }$ $(\nu=1,\cdots,N)$, together with invertible linear maps $%
T_{\nu }:\mathbb{R}^{q}\rightarrow \mathbb{R}^{q}$, such that each $T_{\nu
}E_{\nu }$ is given in the form assumed in Algorithm \ref{algorithm-complexify}. For each $\nu $, Algorithm \ref{algorithm-6.5} also
produces polynomials $P_{\mu \nu }$ $\left( \mu =1,\cdots ,q-n_{\nu }\right) 
$ that play the r\^{o}le of the $P_{\mu }$ in Algorithm \ref{algorithm-complexify}. From Algorithm \ref{algorithm-complexify}, we
therefore obtain for each $\nu $ an affine variety $V_{\nu }\subset \mathbb{%
C}^{q}$ such that for any $P\in \mathbb{C}\left[ z_{1},\cdots ,z_{q}\right] $%
, $P|_{T_{\nu}E_{\nu }}=0$ if and only if $P|_{V_{\nu }}=0$. Algorithm \ref{algorithm-generators-ideal-algebraic-set} produces generators $P_{\nu ,1},\cdots ,P_{\nu ,G\left(
\nu \right) }\in \mathbb{C}\left[ z_{1},\cdots ,z_{q}\right] $ for the ideal
of polynomials in $\mathbb{C}\left[ z_{1},\cdots ,z_{q}\right] $ that vanish
on $V_{\nu }$. Thus, $P_{\nu ,1},\cdots ,P_{\nu ,G\left( \nu \right) }$
generate the ideal of polynomials in $\mathbb{C}\left[ z_{1},\cdots ,z_{q}%
\right] $ that vanish on $T_{\nu} E_{\nu }\subset \mathbb{R}^{q}$. Consequently,
the real and imaginary parts of the $P_{\nu ,g}$ $\left( g=1,\cdots ,G\left(
\nu \right) \right) $ generate the ideal of all polynomials in $\mathbb{R}%
\left[ x_{1},\cdots ,x_{q}\right] $ that vanish on $T_{\nu}E_{\nu }$. It is now
trivial to produce generators for the ideal $\mathcal{I}_{\nu }$ of all
polynomials in $\mathbb{R}\left[ x_{1},\cdots ,x_{q}\right] $ that vanish on 
$E_{\nu }$. Because $E$ is the union of the $E_{\nu }$ the ideal of all
polynomials in $\mathbb{R}\left[ x_{1},\cdots ,x_{q}\right] $ that vanish on 
$E$ is equal to $\mathcal{I}_{1}\cap \cdots \cap \mathcal{I}_{N}$. We have
produced generators for each $\mathcal{I}_{\nu }$. Using Algorithm \ref{Algorithm-solution-to-module-ofPolynmialEq}, we can now produce generators for their
intersection. Thus, we compute generators for the ideal of all real
polynomials that vanish on $E$. This completes our explanation of Algorithm \ref{Algorithm-generators-for-ideal-vanishing-on-semiA}.
\end{proof}

\subsection{} We conclude this section with an elementary algorithm. 

\begin{algorithm}[Find Critical $l$]
\label{critical-l-algorithm}Given matrices $\left( A_{ij}\right)
_{i=1,\cdots ,I;j=1,\cdots J}$ and $\left( B_{ik}\right) _{i=1,\cdots
I;k=1,\cdots ,K}$ with polynomial entries $A_{ij},B_{ik}\in \mathbb{R}\left[
x_{1},\cdots ,x_{n}\right] $; and given a nonzero polynomial $\Delta \left(
x_{1},\cdots ,x_{n}\right) $, we produce an integer $%
l_{0}\geq 0$ with the following property:

Let $P_{1},\cdots ,P_{K}\in \mathbb{R}[x_{1},\cdots ,x_{n}]$. Suppose there
exist $Q_{1},\cdots ,Q_{J}\in \mathbb{R}[x_{1},\cdots ,x_{n}]$ and an
integer $l\geq 0$ such that 
\begin{eqnarray}
&&\left( \Delta \left( x_{1},\cdots ,x_{n}\right) \right)
^{l}\sum_{k=1}^{K}B_{ik}\left( x_{1},\cdots ,x_{n}\right) P_{k}\left(
x_{1},\cdots ,x_{n}\right)  \notag \\
&=&\sum_{j=1}^{J}A_{ij}\left( x_{1},\cdots ,x_{n}\right) Q_{j}\left(
x_{1},\cdots ,x_{n}\right) \text{, for each }i=1,\cdots ,I\text{.}
\label{prelim-sharp-l}
\end{eqnarray}%
Then there exist $Q_{1}^{0},\cdots ,Q_{J}^{0}\in \mathbb{R}\left[
x_{1},\cdots ,x_{n}\right] $ such that 
\begin{eqnarray*}
&&\left( \Delta \left( x_{1},\cdots ,x_{n}\right) \right)
^{l_{0}}\sum_{k=1}^{K}B_{ik}\left( x_{1},\cdots ,x_{n}\right) P_{k}\left(
x_{1},\cdots ,x_{n}\right) \\
&=&\sum_{j=1}^{J}A_{ij}\left( x_{1},\cdots ,x_{n}\right) Q_{j}^{0}\left(
x_{1},\cdots ,x_{n}\right) \text{, for each }i=1,\cdots ,I\text{.}
\end{eqnarray*}
\end{algorithm}

\begin{proof}[\textbf{Explanation}]
We write $\vec{P}$ to denote a vector $(P_{1},\cdots ,P_{K})$ of
polynomials, and we write $\mathcal{M}_{l}$ to denote the $\mathbb{R}%
[x_{1},\cdots ,x_{n}]$-module of all $\vec{P}$ for which %
\eqref{prelim-sharp-l} admits a polynomial solution $Q_{1},\cdots ,Q_{J}$.

Given any $l \geq 0$, we can compute generators $\vec{P}^{g,l}$ for $%
\mathcal{M}_l$ ($g=1,\cdots,G_l$).

We have $\mathcal{M}_{0}\subseteq \mathcal{M}_{1}\subseteq \cdots $. Since $%
\mathbb{R}\left[ x_{1},\cdots ,x_{n}\right] $ is Noetherian, it follows that 
\begin{equation*}
\mathscr{M}_{l_{0}}=\mathscr{M}_{l_{0}+1}=\cdots \text{ for some }l_{0}\text{%
.}
\end{equation*}%
We can test a given $l$ to decide whether $\mathscr{M}_{l}=\mathscr{M}%
_{l+1}, $ by checking whether the generators for $\mathscr{M}_{l+1}$ belong
to the module $\mathscr{M}_{l}$, for which we have computed generators.

Therefore, we may successively test $l=0, l=1, l=2$, etc., to determine
whether $\mathcal{M}_l = \mathcal{M}_{l+1}$. If so, we stop.

We know that the above algorithm will eventually terminate, thus producing
an integer $l_{0}\geq 0$ for which%
\begin{equation}
\mathscr{M}_{l_{0}}=\mathscr{M}_{l_{0}+1}.  \label{prelim-*1}
\end{equation}

For any $l\geq l_{0}$, we have%
\begin{equation}
\vec{P}\in \mathscr{M}_{l}\text{ if and only if }\left( \Delta \left(
x_{1},\cdots ,x_{n}\right) \right) ^{l-l_{0}}\vec{P}\in \mathscr{M}_{l_{0}},
\label{prelim-*2}
\end{equation}%
and 
\begin{equation}
\vec{P}\in \mathscr{M}_{l+1}\text{ if and only if }\left( \Delta \left(
x_{1},\cdots ,x_{n}\right) \right) ^{l-l_{0}}\vec{P}\in \mathscr{M}%
_{l_{0}+1}.  \label{prelim-*3}
\end{equation}

From (\ref{prelim-*1}), (\ref{prelim-*2}), (\ref{prelim-*3}), we conclude
that $\mathscr{M}_{l}=\mathscr{M}_{l+1}$ for $l\geq l_{0}$. Thus, we have
computed $l_{0}$ for which 
\begin{equation*}
\mathscr{M}_{l_{0}}=\mathscr{M}_{l_{0}+1}=\mathscr{M}_{l_{0}+2}=\cdots .
\end{equation*}

This completes the explanation of Algorithm \ref{critical-l-algorithm}.
\end{proof}

\section{Solutions of Differential Equations I}

In this section, we suppose we are given a semialgebraic, connected,
real-analytic submanifold $V\subset \mathbb{R}^{n}\times \mathbb{R}^{m}$. We
assume that 
\begin{equation*}
V\subset \left\{ \left( x_{1},\cdots ,x_{n},y_{1},\cdots ,y_{m}\right) \in 
\mathbb{R}^{n}\times \mathbb{R}^{m}:\tilde{P}_{\mu}\left( x_{1},\cdots
,x_{n},y_{\mu }\right) =0\text{ for }\mu =1,\cdots ,m\right\}
\end{equation*}%
where each $\tilde{P}_{\mu}$ is a nonzero polynomial.

We assume that $x_{1},\cdots ,x_{n}$ serve as local coordinates on some
non-empty relatively open subset of $V$.

We fix an integer $J\geq 1$, and we write $\vec{P}$ to denote a vector 
\begin{equation*}
\left( P_{1},\cdots ,P_{J}\right)
\end{equation*}%
of polynomials $P_{j}\in \mathbb{R}\left[ x_{1},\cdots ,x_{n},y_{1},\cdots
,y_{m}\right] $.

We suppose we are given a linear partial differential operator $L$ of order $%
M$ with polynomial coefficients in $\mathbb{R}\left[ x_{1},\cdots
,x_{n},y_{1},\cdots ,y_{m}\right] $.

We suppose that $L$ acts on vectors $\vec{P}$ and produces scalars, i.e., $L%
\vec{P}$ consists of a single polynomial.

We assume that $L$ involves only $y_{j}$-derivatives; it involves no $%
x_{i}$-derivatives.

In this section, we present the following

\begin{algorithm}
\label{main-algorithm-of-section-2}Given $V,\tilde{P}_{\mu}$, $L$ as above,
we produce generators for the $\mathbb{R}\left[ x_{1},\cdots
,x_{n},y_{1},\cdots ,y_{m}\right] $-module $\mathscr{M}$ consisting of all $%
\vec{P}$ such that 
\begin{equation}
L\left( Q\cdot \vec{P}\right) =0  \label{sol-dfq-1}
\end{equation}%
on $V$ for all $Q\in \mathbb{R}\left[ x_{1},\cdots ,x_{n},y_{1},\cdots ,y_{m}%
\right] $.
\end{algorithm}

\begin{proof}[\textbf{Explanation}]
We compute generators $S_{1},\cdots ,S_{\nu _{\max }}$ for the ideal of
polynomials 
\begin{equation*}
P\in \mathbb{R}\left[ x_{1},\cdots ,x_{n},y_{1},\cdots ,y_{m}\right]
\end{equation*}
that vanish on $V$. (See Algorithm \ref{Algorithm-generators-for-ideal-vanishing-on-semiA}.) Write 
\begin{equation*}
\tilde{P}_{\mu}(x_{1},\cdots ,x_{n},y_{\mu })=a_{\mu }(x_{1},\cdots
,x_{n})y_{\mu }^{D_{\mu }}+\sum_{j<D_{\mu }}b_{\mu j}(x_{1},\cdots
,x_{n})y_{\mu }^{j},
\end{equation*}
with $a_{\mu }$ nonzero.

(Note that $\tilde{P}_{\mu}$ cannot be independent of $y_{\mu }$, since $%
\tilde{P}_{\mu}=0$ on $V$, $\tilde{P}_{\mu}$ is not identically zero, and $%
\left( x_{1},\cdots ,x_{n}\right) $ serve as local coordinates on a
neighborhood in $V$.)

Set $\Delta \left( x_{1},\cdots ,x_{n}\right) =\prod_{\mu =1}^{m}a_{\mu
}\left( x_{1},\cdots ,x_{n}\right) $. Thus $\Delta $ isn't identically zero.

Now suppose $\vec{P}$ satisfies (\ref{sol-dfq-1}). Applying Lemma \ref%
{prelim-lemma2}, we can write 
\begin{equation}
\left( \Delta \left( x_{1},\cdots ,x_{n}\right) \right) ^{l}\vec{P}%
=\sum_{\mu =1}^{m}\left[ \tilde{P}_{\mu}\left( x_{1},\cdots ,x_{n},y_{\mu
}\right) \right] ^{M+1}\vec{H}_{\mu }\left( x_{1},\cdots ,x_{n},y_{1},\cdots
,y_{m}\right) +\vec{P}^{\#},  \label{sol-dfq-2}
\end{equation}%
where $\deg _{y}\vec{P}^{\#}\leq D_{1}^{*}$. Here, $D_{1}^{*}$ may be
computed from $D_{1},\cdots ,D_{m},M$.

Since $\vec{P}\in \mathscr{M}$, we have $\left( \Delta \left( x_{1},\cdots
,x_{n}\right) \right) ^{l}\vec{P}\in \mathscr{M}$.

Also, since $L$ is $M^{th}$-order and $\left[ \tilde{P}_{\mu}\left(
x_{1},\cdots ,x_{n},y_{\mu }\right) \right] ^{M+1}$ vanishes to $\left(
M+1\right) ^{rst}$-order at $V$, we have 
\begin{equation*}
\left[ \tilde{P}_{\mu}\left( x_{1},\cdots ,x_{n},y_{\mu }\right) \right]%
^{(M+1)} \cdot \vec{H}_{\mu }\left( x_{1},\cdots ,x_{n},y_{1},\cdots
,y_{m}\right) \in \mathscr{M}\text{ for each }\mu \text{.}
\end{equation*}%
Hence, (\ref{sol-dfq-2}) tells us that $\vec{P}^{\#}\in \mathscr{M}$. Thus, $%
L\left( QP^{\#}\right) =0$ on $V$ for each polynomial $Q$. In particular,
for each multiindex $\gamma $ of order $\left\vert \gamma \right\vert \leq M$%
, we have 
\begin{equation*}
L\left( y^{\gamma }\vec{P}^{\#}\right) =0\text{ on }V,
\end{equation*}%
and therefore 
\begin{equation*}
L\left( y^{\gamma }\vec{P}^{\#}\right) =\sum_{\nu =1}^{\nu _{\max }}\dot{A}%
_{\nu }^{\gamma }\left( x_{1},\cdots ,x_{n},y_{1},\cdots ,y_{m}\right)
S_{\nu }\left( x_{1},\cdots ,x_{n},y_{1},\cdots ,y_{m}\right)
\end{equation*}%
for polynomials $\dot{A}_{\nu }^{\gamma }\in \mathbb{R}\left[ x_{1},\cdots
,x_{n},y_{1},\cdots ,y_{m}\right] $.

Applying Lemma \ref{prelim-lemma2} to $\dot{A}_{\nu }^{\gamma }$, we see
that we can express 
\begin{eqnarray}
&&\left( \Delta \left( x_{1},\cdots ,x_{n}\right) \right) ^{l^{\prime
}}L\left( y^{\gamma }\vec{P}^{\#}\right)  \notag \\
&=&\sum_{\nu =1}^{\nu _{\max }}\hat{A}_{\nu }^{\gamma }\left( x_{1},\cdots
,x_{n},y_{1},\cdots ,y_{m}\right) S_{\nu }\left( x_{1},\cdots
,x_{n},y_{1},\cdots ,y_{m}\right)  \label{sol-dfq-3} \\
&&+\sum_{\mu =1}^{m}\tilde{P}_{\mu}\left( x_{1},\cdots ,x_{n},y_{\mu
}\right) \tilde{H}_{\mu }^{\gamma }\left( x_{1},\cdots ,x_{n},y_{1},\cdots
,y_{m}\right) ,  \notag
\end{eqnarray}%
where the $\hat{A}_{\nu }^{\gamma }$ are polynomials, the $\tilde{H}_{\mu
}^{\gamma }$ are polynomials, and $\deg _{y}\hat{A}_{\nu }^{\gamma }\leq
D_{2}^*$. Here, $D_{2}^*$ may be computed from $D_{1}^*$ and the $\tilde{P}%
_{\mu}$.

We now conclude from (\ref{sol-dfq-3}) that $\sum_{\mu =1}^{m}\tilde{P}%
_{\mu}\left( x_{1},\cdots ,x_{n},y_{\mu }\right) \tilde{H}_{\mu }^{\gamma
}\left( x_{1},\cdots ,x_{n}\right) $ has degree at most $D_{3}^*$ in $\left(
y_{1},\cdots ,y_{m}\right) $, where $D_{3}^*$ may be computed from from $L$, $%
D_{1}^*,M,D_{2}^*$ and the $S_{\nu }$.

Therefore, by Lemma \ref{prelim-lemma4}, we have (for some $\tilde{l}\geq 0$%
):%
\begin{eqnarray}
&&\left( \Delta \left( x_{1},\cdots ,x_{n}\right) \right) ^{\tilde{l}%
}\sum_{\mu }\tilde{P}_{\mu }\left( x_{1},\cdots ,x_{n},y_{\mu }\right) 
\tilde{H}_{\mu }^{\gamma }\left( x_{1},\cdots ,x_{n},y_{1},\cdots
,y_{m}\right)  \notag \\
&=&\sum_{\mu }\tilde{P}_{\mu }\left( x_{1},\cdots ,x_{n},y_{\mu }\right)
H_{\mu }^{\gamma \#}\left( x_{1},\cdots ,x_{n},y_1, \cdots, y_{m}\right)
\label{sol-dfq-3*}
\end{eqnarray}%
with $\deg _{y}H_{\mu }^{\gamma \#}\leq D_{4}^*$; here, $D_{4}^*$ may be
computed from $D_{3}^*$ and the $\tilde{P}_{\mu }$.

Substituting (\ref{sol-dfq-3*}) into (\ref{sol-dfq-3}), we learn that 
\begin{eqnarray}
&&\left( \Delta \left( x_{1},\cdots ,x_{n}\right) \right) ^{l^{\prime \prime
}}L\left( y^{\gamma }\vec{P}^{\#}\right)  \notag \\
&=&\sum_{\nu =1}^{\nu _{\max }}\check{A}_{\nu }^{\gamma }\left( x_{1},\cdots
,x_{n},y_{1},\cdots ,y_{m}\right) S_{\nu }\left( x_{1},\cdots
,x_{n},y_{1},\cdots ,y_{m}\right)  \label{sol-dfq-4} \\
&&+\sum_{\mu }\tilde{P}_{\mu }\left( x_{1},\cdots ,x_{n},y_{\mu }\right)
H_{\mu }^{\gamma \#}\left( x_{1},\cdots ,x_{n},y_{m}\right) \text{ (each } |\gamma| \leq M)  \notag
\end{eqnarray}%
with $\deg _{y}\vec{P}^{\#}\leq D_{1}^*$, $\deg _{y}\check{A}_{\nu }^{\gamma
}\leq D_{2}^*$, and $\deg _{y}H_{\mu }^{\gamma \#}\leq D_{4}^*$.

In view of these bounds on the degrees of $\vec{P}^{\#}$, $\check{A}_{\nu
}^{\gamma }$, $H_{\mu }^{\gamma \#}$ in $\left( y_{1},\cdots ,y_{m}\right) $,
we may view (\ref{sol-dfq-4}) as a system of linear equations with polynomial
coefficients in $\mathbb{R}\left[ x_{1},\cdots ,x_{n}\right] $; here, we use
the fact that $L$ differentiates only in $\left( y_{1},\cdots ,y_{m}\right)
, $ not in $\left( x_{1},\cdots ,x_{n}\right) $.

We now perform Algorithm \ref{critical-l-algorithm}. Thus, from $L$, $%
S_{1},\cdots ,S_{\nu },P_{1},\cdots ,P_{m},D_{1}^*,D_{2}^*,D_{4}^*$, we
produce an integer $l_{0}\geq 0$ such that (\ref{sol-dfq-4}) implies that we
may express 
\begin{eqnarray}
&&\left( \Delta \left( x_{1},\cdots ,x_{n}\right) \right) ^{l_{0}}L\left(
y^{\gamma }\vec{P}^{\#}\right)  \notag \\
&=&\sum_{\nu =0}^{\nu _{\max }}A_{\nu }^{\gamma }\left( x_{1},\cdots
,x_{n},y_{1},\cdots ,y_{m}\right) S_{\nu }\left( x_{1},\cdots
,x_{n},y_{1},\cdots ,y_{m}\right)  \label{sol-dfq-5} \\
&&+\sum_{\mu }\tilde{P}_{\mu}\left( x_{1},\cdots ,x_{n},y_{\mu }\right)
H_{\mu }^{\gamma }\left( x_{1},\cdots ,x_{n},y_1, \cdots, y_{m}\right) \text{ (each } |\gamma| \leq M) \notag
\end{eqnarray}%
with $\deg _{y}A_{\nu }^{\gamma }\leq D_{2}^*$ and $\deg _{y}H_{\mu
}^{\gamma }\leq D_{4}^*$.

We can compute generators $\vec{P}_{1}^{\#},\cdots ,\vec{P}_{K}^{\#}$ for
the $\mathbb{R}\left[ x_{1},\cdots ,x_{n}\right] $-module of all $\vec{P}%
^{\#}$ with $\deg _{y}\vec{P}^{\#}\leq D_{1}^*$, for which (\ref{sol-dfq-5})
admits a polynomial solution $\left( A_{\nu }^{\gamma }\right) _{\left\vert
\gamma \right\vert \leq M,\nu =1,\cdots ,\nu _{\max }}$, $\left( H_{\mu
}^{\gamma }\right) _{\left\vert \gamma \right\vert \leq M,\mu =1,\cdots ,m}$
with $\deg _{y}A_{\nu }^{\gamma }\leq D_{2}^*$ and $\deg _{y}H_{\mu
}^{\gamma }\leq D_{4}^*$.

Thus, for some polynomials $A_{\nu ,k}^{\gamma }$, $H_{\mu ,k}^{\gamma }\in 
\mathbb{R}\left[ x_{1},\cdots ,x_{n},y_{1},\cdots ,y_{m}\right] $, we have 
\begin{equation}
\Delta ^{l_{0}}L\left( y^{\gamma }\vec{P}^{\#}_k\right) =\sum_{\nu }A_{\nu
,k}^{\gamma }S_{\nu }+\sum_{\mu }H_{\mu ,k}^{\gamma }\tilde{P}_{\mu },\text{
for }\left\vert \gamma \right\vert \leq M,k=1,\cdots ,K;  \label{sol-dfq-6}
\end{equation}%
and 
\begin{equation*}
\vec{P}^{\#}\left( x_{1},\cdots ,x_{n},y_{1},\cdots ,y_{m}\right)
=\sum_{k=1}^{K}G_{k}\left( x_{1},\cdots ,x_{n}\right) \vec{P}_{k}^{\#}\text{
for some polynomials }G_{k}\text{.}
\end{equation*}%
Substituting this last equation into (\ref{sol-dfq-2}), we find that 
\begin{eqnarray}
&&\Delta ^{l}\vec{P}\left( x_{1},\cdots ,x_{n},y_{1},\cdots ,y_{m}\right) 
\notag \\
&=&\sum_{\mu =1}^{m}\left[ \tilde{P}_{\mu }\left( x_{1},\cdots ,x_{n},y_{\mu
}\right) \right] ^{M+1}\vec{H}_{\mu }\left( x_{1},\cdots ,x_{n},y_{1},\cdots
,y_{m}\right)   \label{sol-dfq-7} \\
&&+\sum_{k=1}^{K}G_{k}\left( x_{1},\cdots ,x_{n},y_{1},\cdots ,y_{m}\right) 
\vec{P}_{k}^{\#}  \notag
\end{eqnarray}%
for some polynomials $\vec{H}_{\mu }$, $G_{k}\in \mathbb{R}\left[
x_{1},\cdots ,x_{n},y_{1},\cdots ,y_{m}\right] $. (We forget that the $G_{k}$
depend only on $x_{1},\cdots ,x_{n}$.)

This holds for some $l$. Using Algorithm \ref{critical-l-algorithm}, we can
compute from the $\tilde{P}_{\mu }$ and $\vec{P}_{k}^{\#}$ (and $M$) an integer $%
l_{1}\geq 0$ such that (\ref{sol-dfq-7}) implies the existence of $\vec{H}%
_{\mu }^{\ast }$ and $G_{k}^{\ast }\in \mathbb{R}\left[ x_{1},\cdots
,x_{n},y_{1},\cdots ,y_{m}\right] $ satisfying 
\begin{eqnarray}
&&\Delta ^{l_{1}}\vec{P}\left( x_{1},\cdots ,x_{n},y_{1},\cdots ,y_{m}\right)
\notag \\
&=&\sum_{\mu =1}^{m}\left[ \tilde{P}_{\mu }\left( x_{1},\cdots
,x_{n},y_{\mu }\right) \right] ^{M+1}\vec{H}_{\mu }^{\ast }\left( x_{1},\cdots
,x_{n},y_{1},\cdots ,y_{m}\right)   \label{sol-dfq-8} \\
&&+\sum_{k=1}^{K}G_{k}^{\ast }\left( x_{1},\cdots ,x_{n},y_{1},\cdots
,y_{m}\right) \vec{P}_{k}^{\#}\left( x_{1},\cdots ,x_{n},y_{1},\cdots
,y_{m}\right) .  \notag
\end{eqnarray}

To summarize: Suppose $\vec{P}$ belongs to the module $\mathscr{M}$, i.e.,
suppose $L\left( Q\vec{P}\right) =0$ on $V$ for all $Q\in \mathbb{R}\left[
x_{1},\cdots ,x_{n},y_{1},\cdots ,y_{m}\right] $. Then there exist $\vec{H}%
_{\mu }^{\ast }$ and $G_{k}^{\ast }\in \mathbb{R}\left[ x_{1},\cdots
,x_{n},y_{1},\cdots ,y_{m}\right] $ satisfying (\ref{sol-dfq-8}). Here, the
integer $l_{1}$ and the polynomial vectors $\vec{P}_{k}^{\#}$ ($k=1,\cdots
,K $) have been computed from $V$, $\tilde{P}_{\mu }$ ($\mu =1,\cdots ,m$),
and $M$, $L$.

Conversely, suppose (\ref{sol-dfq-8}) holds. We will prove that $\vec{P}$
belongs the module $\mathscr{M}$, i.e., $L\left( Q\vec{P}\right) =0$ on $V$
for all $Q\in \mathbb{R}\left[ x_{1},\cdots ,x_{n},y_{1},\cdots ,y_{m}\right]
$.

We know that $\left[ \tilde{P}_{\mu }\left( x_{1},\cdots
,x_{n},y_{\mu }\right) \right] ^{M+1}\vec{H}_{\mu }^{\ast }\left( x_{1},\cdots
,x_{n},y_{1},\cdots ,y_{m}\right) $ belongs to $\mathscr{M}$, since $L$
is of $M^{th}$-order and $\left[ \tilde{P}_{\mu }\left( x_{1},\cdots
,x_{n},y_{\mu }\right) \right] ^{M+1}$ vanishes to $\left( M+1\right) ^{%
\text{rst}}$ order on $V$.

We will check that 
\begin{equation}
L\left( Q\vec{P}_{k}^{\#}\right) =0\text{ on }V\text{, for }k=1,\cdots ,K,%
\text{ and each }Q\in \mathbb{R}\left[ x_{1},\cdots ,x_{n},y_{1},\cdots
,y_{m}\right] \text{.}  \label{sol-dfq-9}
\end{equation}%
Once we prove (\ref{sol-dfq-9}), we will know that $G_{k}^{\ast }\vec{P}%
_{k}^{\#}\in \mathscr{M}$. In view of (\ref{sol-dfq-8}), that will tell us
that $\left( \Delta \left( x_{1},\cdots ,x_{n}\right) \right) ^{l_{1}}\vec{P}%
\left( x_{1},\cdots ,x_{n},y_{1},\cdots ,y_{m}\right) \in \mathscr{M}$,
i.e., 
\begin{equation*}
L\left( Q\Delta ^{l_{1}}\vec{P}\right) =0\text{ on }V,\text{ for each }Q\in 
\mathbb{R}\left[ x_{1},\cdots ,x_{n},y_{1},\cdots ,y_{m}\right] \text{.}
\end{equation*}%
On the other hand, since $\Delta =\Delta \left( x_{1},\cdots ,x_{n}\right) $
and $L$ involves no $x$-derivatives, we have 
\begin{equation*}
L\left( Q\Delta ^{l_{1}}\vec{P}\right) =\Delta ^{l_{1}}L\left( Q\vec{P}%
\right) ,
\end{equation*}%
and thus 
\begin{equation*}
\Delta ^{l_{1}}L\left( Q\vec{P}\right) =0\text{ on }V\text{ for each }Q\in 
\mathbb{R}\left[ x_{1},\cdots ,x_{n},y_{1},\cdots ,y_{m}\right] \text{.}
\end{equation*}%
By Lemma \ref{prelim-lemma1}, we will know that $L\left( Q\vec{P}\right) =0$
on $V$ for each $Q\in \mathbb{R}\left[ x_{1},\cdots ,x_{n},y_{1},\cdots
,y_{m}\right] $, completing the proof that $\vec{P}\in \mathscr{M}$ as
claimed.

Thus, once we check (\ref{sol-dfq-9}), we will know that $\vec{P}\in %
\mathscr{M}$ if and only if (\ref{sol-dfq-8}) admits a polynomial solution $%
\left( H_{\mu }^{\ast },G_{k}^{\ast }\right) $.

To prove (\ref{sol-dfq-9}), we return to (\ref{sol-dfq-6}). We know that the
polynomials $S_{\nu }$, $\tilde{P}_{\mu}$ vanish on $V$. Hence, (\ref%
{sol-dfq-6}) tells us that 
\begin{equation*}
\Delta ^{l_{0}}L\left( y^{\gamma }\vec{P}_{k}^{\#}\right) =0\text{ on }V%
\text{ (any }\left\vert \gamma \right\vert \leq M\text{).}
\end{equation*}%
Again applying Lemma \ref{prelim-lemma1}, we see that 
\begin{equation}
L\left( Q\vec{P}_{k}^{\#}\right) =0\text{ on }V\text{ for any polynomial }Q%
\text{ of degree at most }M\text{.}  \label{sol-dfq-10}
\end{equation}%
(Here, we use the fact that $L$ involves no $x$-derivatives; multiplying $%
\vec{P}_{k}^{\#}$ by $x^{\beta }$ has a trivial effect.)

Now, let $Q$ be any polynomial, and let $\left( x^{o},y^{o}\right) \in V$.
We expand $Q$ about $(x^{o},y^{o})$ and thus write 
\begin{equation*}
Q=Q_{low}+\sum_{\left\vert \beta \right\vert +\left\vert \gamma \right\vert
\geq M+1}Q_{\beta \gamma }\left( x-x^{o}\right) ^{\beta }\left(
y-y^{o}\right) ^{\gamma }
\end{equation*}%
for coefficients $Q_{\beta \gamma }$ and a polynomial $Q_{low}$, with $\deg
\left( Q_{low}\right) \leq M$.

We have already seen that $L\left( Q_{low}\vec{P}_{k}^{\#}\right) =0$ on $V$
(see (\ref{sol-dfq-10})), hence $L\left( Q_{low}\vec{P}_{k}^{\#}\right) $ vanishes
at $\left( x^{o},y^{o}\right) $.

On the other hand, since $L$ is of $M^{th}$-order, we have also 
\begin{equation*}
L\left( Q_{\beta \gamma }\left( x-x^{o}\right) ^{\beta }\left(
y-y^{o}\right) ^{\gamma }\vec{P}_{k}^{\#}\right) =0\text{ at }\left(
x^{o},y^{o}\right) ,\text{ for }\left\vert \beta \right\vert +\left\vert
\gamma \right\vert \geq M+1.
\end{equation*}%
Therefore, $L\left( Q\vec{P}_{k}^{\#}\right) =0$ at $\left(
x^{o},y^{o}\right) $ for any $Q\in \mathbb{R}\left[ x_{1},\cdots
,x_{n},y_{1},\cdots ,y_{m}\right] $, completing the proof of (\ref{sol-dfq-9}%
).

We now know that $\vec{P}\in \mathscr{M}$ if and only if there exist
polynomials $H_{\mu }^{\ast }\left( x_{1},\cdots ,x_{n},y_{1},\cdots
,y_{m}\right) $ ($\mu =1,\cdots ,m$) and $G_{k}^{\ast }\left( x_{1},\cdots
,x_{n},y_{1},\cdots ,y_{m}\right) $ ($k = 1,\cdots, K$) satisfying (\ref{sol-dfq-8}).

Since $l_{1}$, $\tilde{P}_{\mu}$, and $\vec{P}_{k}^{\#}$ in (\ref{sol-dfq-8}%
) have all been computed, we may now produce generators for the $\mathbb{R}%
\left[ x_{1},\cdots ,x_{n},y_{1},\cdots ,y_{m}\right] $-module of solutions $%
\left( \vec{P},\left( \vec{H}_{\mu }^{\ast }\right) _{\mu =1,\cdots
,m}\left( G_{k}^{\ast }\right) _{k=1,\cdots ,K}\right) $ of (\ref{sol-dfq-8}%
).

The $\vec{P}$-components of these generators are generators for the module $%
\mathscr{M}$.

This completes our explanation of Algorithm \ref{main-algorithm-of-section-2}%
.
\end{proof}

\section{Solutions of Differential Equations II}

In this section, we assume that 
\begin{equation*}
V\subset \left\{ 
\begin{array}{c}
\left( x_{1},\cdots ,x_{n},y_{1},\cdots ,y_{m},z_{1},\cdots ,z_{p}\right)
\in \mathbb{R}^{n}\times \mathbb{R}^{m}\times \mathbb{R}^{p}: \\ 
\tilde{P}_{\mu }\left( x_{1},\cdots ,x_{n},y_{\mu }\right) =0\text{ }\left(
\mu =1,\cdots ,m\right) , \\ 
\hat{P}_{\lambda }\left( x_{1},\cdots ,x_{n},z_{\lambda }\right) =0\text{ }%
\left( \lambda =1,\cdots ,p\right) 
\end{array}%
\right\} 
\end{equation*}%
where $\tilde{P}_{\mu },\hat{P}_{\lambda }$ are nonzero polynomials.

We assume that $V\subset \mathbb{R}^{n}\times \mathbb{R}^{m}\times \mathbb{R}%
^{p}$ is a semialgebraic real-analytic connected manifold, and that $\left(
x_{1},\cdots ,x_{n}\right) $ serve as local coordinates on a nonempty
(relatively) open subset of $V$.

We fix an integer $J\geq 1$, and we write $\vec{P}$ to denote a vector $%
\left( P_{1},\cdots ,P_{J}\right) $ of polynomials 
\begin{equation*}
P_{j}\in \mathbb{R}\left[ x_{1},\cdots ,x_{n},y_{1},\cdots
,y_{m},z_{1},\cdots ,z_{p}\right] .
\end{equation*}

We suppose we are given an $M^{th}$-order linear partial differential
operator $L$ with polynomial coefficients in $\mathbb{R}\left[ x_{1},\cdots
,x_{n},y_{1},\cdots ,y_{m},z_{1},\cdots ,z_{p}\right] $.

We suppose $L$ maps vectors $\vec{P}$ to scalars (i.e., $L\vec{P}$ is a
single polynomial).

We assume that $L$ contains no $x$-derivatives, although it may involve both 
$y$ and $z$-derivatives. 

We write $\vec{P}\left( x,y\right) $ to denote a vector of polynomials $%
\left( P_{1},\cdots ,P_{J}\right) $, with each 
\begin{equation*}
P_{j}\in \mathbb{R}[ x_{1},\cdots ,x_{n},y_{1},\cdots ,y_{m}]
\end{equation*}
(i.e., the $P_{j}$ do not depend on $z_{1},\cdots ,z_{p}$).

In this section, we present the following

\begin{algorithm}
\label{main-algorithm-section3}Given $V,\tilde{P}_{\mu },\hat{P}_{\lambda }$%
, $L$ as above, we compute generators for the $\mathbb{R}\left[ x_{1},\cdots
,x_{n},y_{1},\cdots ,y_{m}\right] $ module consisting of all vectors $\vec{P}%
\left( x,y\right) $ such that 
\begin{equation}
L\left( Q\vec{P}\right) =0\text{ on }V\text{ for every }%
Q\in \mathbb{R}\left[ x_{1},\cdots ,x_{n},y_{1},\cdots ,y_{m},z_{1},\cdots
,z_{p}\right] \text{.}  \label{section3-1}
\end{equation}
\end{algorithm}

\begin{proof}[\textbf{Explanation}]
Applying Algorithm \ref{main-algorithm-of-section-2} (with $z_{1},\cdots
,z_{p}$ regarded as $y_{m+1},\cdots ,y_{m+p}$), we obtain generators $\vec{P}%
_{1},\cdots ,\vec{P}_{K}$ for the $\mathbb{R}\left[ x_{1},\cdots
,x_{n},y_{1},\cdots ,y_{m},z_{1},\cdots ,z_{p}\right] $-module of all $\vec{P%
}$ such that $L\left( Q\vec{P}\right) =0$ on $V$ for all $Q\in \mathbb{R}%
\left[ x_{1},\cdots ,x_{n},y_{1},\cdots ,y_{m},z_{1},\cdots ,z_{p}\right] $.

We write $D_{1}^{*},D_{2}^{*},\cdots $ to denote constants that can be computed from 
$V,\tilde{P}_{\mu },\hat{P}_{\lambda },L$.

For each $\lambda =1,\cdots ,p$, we write 
\begin{equation*}
\hat{P}_{\lambda }\left( x_{1},\cdots ,x_{n},z_{\lambda }\right) =a_{\lambda
}\left( x_{1},\cdots ,x_{n}\right) z_{\lambda }^{D_{\lambda
}}+\sum_{j=0}^{D_{\lambda }-1}b_{\lambda j}\left( x_{1},\cdots ,x_{n}\right)
\left( z_{\lambda }\right) ^{j},
\end{equation*}%
with $a_{\lambda }\left( x_{1},\cdots ,x_{n}\right) \not\equiv 0$. (Note
that $\hat{P}_{\lambda }$ cannot be independent of $z_{\lambda }$, since $%
\hat{P}_{\lambda }=0$ on $V$, yet $\hat{P}_{\lambda }\not\equiv 0$, and $\left( x_{1},\cdots ,x_{n}\right) $ serve as local coordinates on
a nonempty neighborhood in the real-analytic manifold $V$.)

Let $\Delta \left( x_{1},\cdots ,x_{n}\right) =\prod_{\lambda
=1}^{p}a_{\lambda }\left( x_{1},\cdots ,x_{n}\right) $; thus $\Delta $ is a
nonzero polynomial.

Now suppose $\vec{P}\left( x,y\right) $ satisfies (\ref{section3-1}). Then
for polynomials 
\begin{equation*}
A_{k}^{\left( 0\right) }\left( x_{1},\cdots ,x_{n},y_{1},\cdots
,y_{m},z_{1},\cdots ,z_{p}\right) 
\end{equation*}%
($k=1,\cdots ,K$), we have 
\begin{equation}
\vec{P}\left( x,y\right) =\sum_{k=1}^{K}A_{k}^{\left( 0\right) }\vec{P}_{k}.
\label{section3-2}
\end{equation}%
Note that the $A_{k}^{\left( 0\right) }$ and $\vec{P}_{k}$ may depend on $%
z_{1},\cdots ,z_{p}$, although $\vec{P}\left( x,y\right) $ does not. By
Lemma \ref{prelim-lemma2}, there exist $l\geq 0$ and polynomials $%
A_{k}^{\left( 1\right) }$, $B_{k\lambda }\in \mathbb{R}\left[ x_{1},\cdots
,x_{n},y_{1},\cdots ,y_{m},z_{1},\cdots ,z_{p}\right] $ such that 
\begin{eqnarray*}
&&\left( \Delta \left( x_{1},\cdots ,x_{n}\right) \right) ^{l}A_{k}^{\left(
0\right) }\left( x_{1},\cdots ,x_{n},y_{1},\cdots ,y_{m},z_{1},\cdots
,z_{p}\right)  \\
&=&\sum_{\lambda =1}^{p}B_{k\lambda }\left( x_{1},\cdots ,x_{n},y_{1},\cdots
,y_{m},z_{1},\cdots ,z_{p}\right) \left[ \hat{P}_{\lambda }\left(
x_{1},\cdots ,x_{n},z_{\lambda }\right) \right] ^{M+1} \\
&&+A_{k}^{\left( 1\right) }\left( x_{1},\cdots ,x_{n},y_{1},\cdots
,y_{m},z_{1},\cdots ,z_{p}\right) 
\end{eqnarray*}%
for each $k=1,\cdots ,K$; and $\deg _{z}A_{k}^{\left( 1\right) }\leq D_{1}^*$.
Substituting this equation into (\ref{section3-2}), we see that 
\begin{equation}
\left( \Delta \left( x_{1},\cdots ,x_{n}\right) \right) ^{l}\vec{P}%
\left( x,y\right) =\sum_{k=1}^{K}A_{k}^{\left( 1\right) }\vec{P}%
_{k}+\sum_{\lambda =1}^{p}\left[ \hat{P}%
_{\lambda }\left( x_{1},\cdots ,x_{n},z_{\lambda }\right) \right] ^{M+1}\cdot \vec{H}_{\lambda }^{\left( 1\right) }
\label{section3-3}
\end{equation}%
for some vector of polynomials $\vec{H}_{\lambda }^{\left( 1\right) }$ in $%
\mathbb{R}\left[ x_{1},\cdots ,x_{n},y_{1},\cdots ,y_{m},z_{1},\cdots ,z_{p}%
\right] $.

Thus, (\ref{section3-3}) holds, with $\deg _{z}A_{k}^{\left( 1\right) }\leq
D_{1}^*$ for all $k$.

All terms in (\ref{section3-3}) except for the sum on $\lambda $ have degree 
$\leq D_{2}^*$ in $\left( z_{1},\cdots ,z_{p}\right) $.

Therefore, (\ref{section3-3}) tells us that 
\begin{equation*}
\deg _{\left( z_{1},\cdots ,z_{p}\right) }\sum_{\lambda =1}^{p}\vec{H}%
_{\lambda }^{\left( 1\right) }\left[ \hat{P}_{\lambda }\left( x_{1},\cdots
,x_{n},z_{\lambda }\right) \right] ^{M+1}\leq D_{2}^*.
\end{equation*}%
Hence, by Lemma \ref{prelim-lemma4}, there exist an integer $l^{\prime }\geq
0$ and vectors of polynomials $\vec{H}_{\lambda }^{\left( 2\right) }$in 
\begin{equation*}
\mathbb{R[}x_{1},\cdots ,x_{n},y_{1},\cdots ,y_{m},z_{1},\cdots ,z_{p}]
\end{equation*}
satisfying 
\begin{equation*}
\left( \Delta \left( x_{1},\cdots ,x_{n}\right) \right) ^{l^{\prime
}}\sum_{\lambda =1}^{p}\left[ \hat{P}%
_{\lambda }\left( x_{1},\cdots ,x_{n},z_{\lambda }\right) \right]
^{M+1}\cdot \vec{H}_{\lambda }^{\left( 1\right) }=\sum_{\lambda =1}^{p}\left[ \hat{%
P}_{\lambda }\left( x_{1},\cdots ,x_{n},z_{\lambda }\right) \right] ^{M+1} \cdot \vec{H}_{\lambda }^{\left( 2\right) }
\end{equation*}%
and $\deg _{\left( z_{1},\cdots ,z_{p}\right) }\vec{H}_{\lambda }^{\left(
2\right) }\leq D_{3}^*$.

Substituting the above equation into (\ref{section3-3}), we find that for
some $l^{\prime \prime }\geq 0$: 
\begin{equation}
\left( \Delta \left( x_{1},\cdots ,x_{n}\right) \right) ^{l^{\prime \prime }}%
\vec{P}\left( x,y\right) =\sum_{k=1}^{K}A_{k}^{\left( 2\right) }\vec{P}%
_{k}+\sum_{\lambda =1}^{p}\left[ \hat{P}_{\lambda }\left( x_{1},\cdots
,x_{n},z_{\lambda }\right) \right] ^{M+1}\vec{H}_{\lambda }^{\left( 2\right)
}\text{.}  \label{section3-4}
\end{equation}%
Here, the $A_{k}^{\left( 2\right) }$, and the components of the vectors $%
\vec{H}_{\lambda }^{\left( 2\right) }$, are polynomials in $x_{1},\cdots
,x_{n}$, $y_{1},\cdots ,y_{m}$, $z_{1},\cdots ,z_{p}$; and these polynomials
have degree at most $D_{4}^*$ in $\left( z_{1},\cdots ,z_{p}\right) $ (thanks
to the above bound for the degrees of $A_{k}^{\left( 1\right) }$, $\vec{H}%
_{\lambda }^{\left( 2\right) }$ in $\left( z_{1},\cdots ,z_{p}\right) $).

We may regard (\ref{section3-4}) as a system of linear equations with
coefficients in 
\begin{equation*}
\mathbb{R}\left[ x_{1},\cdots ,x_{n},y_{1},\cdots ,y_{m}\right] .
\end{equation*}
Applying Algorithm \ref{critical-l-algorithm}, we now produce an integer $%
l_{0}\geq 0$ such that (\ref{section3-4}) implies that 
\begin{equation}
\left( \Delta \left( x_{1},\cdots ,x_{n}\right)\right) ^{l_{0}} \vec{P}\left(
x,y\right) =\sum_{k=1}^{K}A_{k}\vec{P}_{k}+\sum_{\lambda =1}^{p}%
\left[ \hat{P}_{\lambda }\left( x_{1},\cdots ,x_{n},z_{\lambda }\right) %
\right] ^{M+1}\vec{H}_{\lambda }  \label{section3-5}
\end{equation}%
for some $A_{k}$, $\vec{H}_{\lambda }$; where $A_{k}$ and the components of $%
\vec{H}_{\lambda }$ are polynomials in $x_{1},\cdots ,x_{n}$, $y_{1},\cdots
,y_{m}$, $z_{1},\cdots ,z_{p}$ having degree at most $D_{4}^*$ in $\left(
z_{1},\cdots ,z_{p}\right) $.

To summarize: We have computed $l_{0},\Delta ,\vec{P}_{k},\hat{P}_{\lambda }$.
Moreover, whenever $\vec{P}(x,y)$ satisfies $L\left( Q\vec{P}\right) =0$ on $V$ for
all $Q\in \mathbb{R}\left[ x_{1},\cdots ,x_{n},y_{1},\cdots
,y_{m},z_{1},\cdots ,z_{p}\right] $, then (\ref{section3-5}) holds with $A_k$
and $\vec{H}_\lambda$ having degree at most $D_4^*$ in $z$.

Conversely, suppose (\ref{section3-5}) holds. Let $Q\in \mathbb{R}\left[
x_{1},\cdots ,x_{n},y_{1},\cdots ,y_{m},z_{1},\cdots ,z_{p}\right] $ be
given. By the defining property of the $\vec{P}_{k}$, we have 
\begin{equation*}
L\left( QA_{k}\vec{P}_{k}\right) =0\text{ on }V\text{ (each }k\text{).}
\end{equation*}%
Moreover, 
\begin{equation}
L\left( Q\left[ \hat{P}_{\lambda }\left( x_{1},\cdots ,x_{n},z_{\lambda
}\right) \right] ^{M+1}\vec{H}_{\lambda }\right) =0\text{ on }V\text{ (each }%
\lambda \text{),}  \label{section3-6}
\end{equation}%
since $L$ is of order $M$ and $\left[ \hat{P}_{\lambda }\left(
x_{1},\cdots ,x_{n},z_{\lambda }\right) \right] ^{M+1}$ vanishes to $\left(
M+1\right) ^{rst}$-order on $V$.

Therefore, (\ref{section3-5}) yields%
\begin{equation*}
L\left( Q\left( \Delta \left( x_{1},\cdots ,x_{n}\right) \right) ^{l_{0}}%
\vec{P} \right) =0\text{ on }V\text{.}
\end{equation*}%
However, since $L$ involves no $x$-derivatives, we have 
\begin{equation*}
L\left( Q\left( \Delta \left( x_{1},\cdots ,x_{n}\right) \right) ^{l_{0}}%
\vec{P} \right) =\left( \Delta \left( x_{1},\cdots
,x_{n}\right) \right) ^{l_{0}}L\left( Q\vec{P} \right) ,
\end{equation*}%
and therefore 
\begin{equation*}
\left( \Delta \left( x_{1},\cdots ,x_{n}\right) \right) ^{l_{0}}L\left( Q%
\vec{P} \right) =0\text{ on }V.
\end{equation*}%
Applying Lemma \ref{prelim-lemma1}, we conclude that $L\left( Q\vec{P} \right) =0$ on $V$. This holds for arbitrary $Q\in \mathbb{R}%
\left[ x_{1},\cdots ,x_{n},y_{1},\cdots ,y_{m},z_{1},\cdots ,z_{p}\right] $.

Thus, $\vec{P}\left( x,y\right) $ satisfies (\ref{section3-1}) if and only
if it satisfies (\ref{section3-5}) with $A_k$ and $\vec{H}_\lambda$ having
degree at most $D_4^*$ in $z$.

However, using the standard algorithms in Section \ref%
{section-known-algorithms}, we can compute generators for the $\mathbb{R[}%
x_{1},\cdots ,x_{n},y_{1},\cdots ,y_{m}]$-module of all solutions of (\ref%
{section3-5}) with $A_k$ and $\vec{H}_\lambda$ having degree at most $D_4^*$
in $z$.

This concludes our explanation of Algorithm \ref{main-algorithm-section3}.
\end{proof}

\section{Solutions of Differential Equations III}

In this section, we suppose we are given a semialgebraic connected real analytic submanifold $V$ of $\mathbb{R}^{n}\times \mathbb{R}^{m}$, with 
\begin{equation*}
V\subset \left\{ \left( x_{1},\cdots ,x_{n},y_{1},\cdots ,y_{m}\right) \in 
\mathbb{R}^{n}\times \mathbb{R}^{m}:\tilde{P}_{\mu }\left( x_{1},\cdots
,x_{n},y_{\mu }\right) =0\text{ }\left( \mu =1,\cdots ,m\right) \right\} ,
\end{equation*}%
where each $\tilde{P}_{\mu }$ is a nonzero polynomial.

We assume that $\left( x_{1},\cdots ,x_{n}\right) $ serve as local
coordinates on some nonempty relatively open subset $U$ of $V$.

We fix $J\geq 1$ and write $\vec{P}$ to denote a vector $\left( P_{1},\cdots
,P_{J}\right) $ of polynomials 
\begin{equation*}
P_{j}\in \mathbb{R}\left[ x_{1},\cdots ,x_{n},y_{1},\cdots ,y_{m}\right] .
\end{equation*}

We suppose we are given a linear differential operator $L$ of order $\leq M$,
with polynomial coefficients in $\mathbb{R}\left[ x_{1},\cdots
,x_{n},y_{1},\cdots ,y_{m}\right] $.

We assume that $L$ acts on vectors $\vec{P}$ and produces scalars (i.e., $L%
\vec{P}$ consists of a single polynomial).

We do not assume that $L$ involves only $y$-derivatives. Our $L$ may involve
differentiations in any (or all) of the variables $x_{1},\cdots
,x_{n},y_{1},\cdots ,y_{m}$.

The result of this section is the following.

\begin{algorithm}
\label{main-algorithm-section4}Given $V,\tilde{P}_{\mu },L$, we compute
finitely many linear differential operators $L_{1},\cdots ,L_{N}$, with the
following properties

\begin{itemize}
\item Each $L_\nu$ maps vectors $\vec{P}$ to single polynomials.

\item The coefficients of each $L_\nu$ are polynomials in $\mathbb{R}%
[x_1,\cdots,x_n,y_1,\cdots,y_m]$.

\item Each $L_{\nu }$ involves only $y$-differentiations (i.e., no $x$%
-differentiations appear in $L_{\nu }$).

\item Let $\vec{P}$ be given. Then 
\begin{equation*}
L\left( Q\vec{P}\right) =0\text{ on }V\text{ for any }Q\in \mathbb{R}\left[
x_{1},\cdots ,x_{n},y_{1},\cdots ,y_{m}\right]
\end{equation*}%
if and only if 
\begin{equation*}
L_{\nu }\left( Q\vec{P}\right) =0\text{ on }V\text{ for any }Q\in \mathbb{R}%
\left[ x_{1},\cdots ,x_{n},y_{1},\cdots ,y_{m}\right] ,\nu =1,\cdots ,N\text{%
.}
\end{equation*}
\end{itemize}
\end{algorithm}

\begin{proof}[\textbf{Explanation}]
As noted (twice) before, $\tilde{P}_{\mu }\left( x_{1},\cdots ,x_{n},y_{\mu
}\right) $ cannot be independent of $y_{\mu }$, since $\tilde{P}_{\mu }=0$
on $V$ but $\tilde{P}_{\mu }$ is not identically zero, and $\left(
x_{1},\cdots ,x_{n}\right) $ serve as local coordinates in a neighborhood in 
$V$.

We begin by possibly modifying the $\tilde{P}_{\mu }$ as in our explanation of Algorithm \ref{algorithm-complexify}, so that (for each $%
\mu $) the polynomial%
\begin{equation*}
\partial _{y_{\mu }}\tilde{P}_{\mu }\left( x_{1},\cdots ,x_{n},y_{\mu
}\right)
\end{equation*}%
does not vanish everywhere on $V$.

Now let 
\begin{equation*}
\Delta \left( x_{1},\cdots ,x_{n},y_{1},\cdots ,y_{m}\right) =\prod_{\mu
=1}^{m}\partial _{y_{\mu }}\tilde{P}_{\mu }\left( x_{1},\cdots ,x_{n},y_{\mu
}\right) .
\end{equation*}%
Thus, $\Delta $ is a polynomial and $\Delta $ is nonzero somewhere on $V$.

We introduce some useful vector fields on $\mathbb{R}^n \times \mathbb{R}^m$%
: For $j=1,\cdots ,n$, define 
\begin{equation}
X_{j}=\Delta \left( x_{1},\cdots ,x_{n},y_{1},\cdots ,y_{m}\right) \left[
\partial _{x_{j}}-\sum_{\mu =1}^{m}\frac{\partial _{x_{j}}\tilde{P}_{\mu
}\left( x_{1},\cdots ,x_{n},y_{\mu }\right) }{\partial _{y_{\mu }}\tilde{P}%
_{\mu }\left( x_{1},\cdots ,x_{n},y_{\mu }\right) }\partial _{y_{\mu }}%
\right] .  \label{sec4-1}
\end{equation}%
Note that $X_{j}$ has polynomial coefficients. More precisely, 
\begin{eqnarray}
X_{j} &=&\Delta \left( x_{1},\cdots ,x_{n},y_{1},\cdots ,y_{m}\right)
\partial _{x_{j}}+\sum_{\mu =1}^{m}b_{j\mu }\left( x_{1},\cdots
,x_{n},y_{1},\cdots ,y_{m}\right) \partial _{y_{\mu }}  \notag \\
&\equiv &\Delta \left( x,y\right) \partial _{x_{j}}+Y_{j},  \label{sec4-2}
\end{eqnarray}%
where the $b_{j\mu }$ are polynomials and $Y_{j}$ involves only $y$%
-derivatives.

Note also that $X_{j}\tilde{P}_{\mu ^{\prime }}\left( x_{1},\cdots
,x_{n},y_{\mu ^{\prime }}\right) =0$ for each $\mu ^{\prime }$ and each $j$,
since 
\begin{equation*}
\partial _{y_{\mu }}\tilde{P}_{\mu ^{\prime }}\left( x_{1},\cdots
,x_{n},y_{\mu ^{\prime }}\right) =0\text{ if }\mu \not=\mu ^{\prime }.
\end{equation*}

It follows that the $X_{j}$ are tangent to $V$ at all points of $U\setminus
\left\{ \Delta =0\right\} $. (See the beginning of the section for the
definition of $U$.) Since $V$ is a connected real-analytic manifold, and
since $\Delta |_{V}$ is real-analytic and not identically zero, we conclude
that $V \setminus \left\{\Delta=0 \right\}$ is dense in $V$ and $U\setminus \left\{ \Delta =0\right\} $ is dense in $U$. Therefore,
since $V$ and $X_{j}$ are real-analytic, it follows that $X_{j}$ is tangent
to $V$ everywhere on $U$, hence everywhere on $V$. 

We record the observation:%
\begin{equation}
\text{Each }X_{j}\text{ is tangent to }V.  \label{sec4-3}
\end{equation}%

We will be commuting multiplication by $\Delta ^{s}$ past products $%
X_{j_{1}}\cdots X_{j_{t}}$. An easy induction on $t$ gives 
\begin{equation}
\left[ \Delta ^{s},X_{j_{1}}\cdots X_{j_{t}}\right] =\Delta
^{s-t}L_{s,j_{1}\cdots j_{t}}\text{ for }s>t\text{ (}s,t\text{ positive
integers)}  \label{xxxx1}
\end{equation}%
where $L_{s,j_{1}\cdots j_{t}}$ is a differential operator with polynomial
coefficients and order less than $t$. To see this, we write 
\begin{eqnarray*}
\Delta ^{s}X_{j_{1}}\cdots X_{j_{t+1}} &=&X_{j_{1}}\Delta
^{s}X_{j_{2}}\cdots X_{j_{t+1}}-s\Delta ^{s-1}\left( X_{j_{1}}\Delta \right)
X_{j_{2}}\cdots X_{j_{t+1}} \\
&=&X_{j_{1}}\left\{ X_{j_{2}}\cdots X_{j_{t+1}}\Delta ^{s}+\Delta
^{s-t}L_{s,j_{2}\cdots j_{t+1}}\right\} -s\Delta ^{s-1}\left(
X_{j_{1}}\Delta \right) X_{j_{2}}\cdots X_{j_{t+1}} \\
&=&X_{j_{1}}\cdots X_{j_{t+1}}\Delta ^{s}+\left( s-t\right) \Delta
^{s-\left( t+1\right) }\left( X_{j_{1}}\Delta \right) L_{s,j_{2}\cdots
j_{t+1}} \\
&&+\Delta ^{s-t}X_{j_{1}}L_{s,j_{2}\cdots j_{t+1}}-s\Delta ^{s-1}\left(
X_{j_{1}}\Delta \right) X_{j_{2}}\cdots X_{j_{t+1}}
\end{eqnarray*}%
so that 
\begin{eqnarray*}
L_{s,j_{1}\cdots j_{t+1}} &=&\left( s-t\right) \left( X_{j_{1}}\Delta
\right) L_{s,j_{2}\cdots j_{t+1}}+\Delta X_{j_{1}}L_{s,j_{2}\cdots j_{t+1}}
\\
&&-s\Delta ^{t}\left( X_{j_{1}}\Delta \right) X_{j_{2}}\cdots X_{j_{t+1}}%
\text{.}
\end{eqnarray*}

Since $L$ is a linear differential operator of order at most $M$ with
polynomial coefficients, we may write 
\begin{equation}
L\left( \vec{P}\right) =\sum_{\left\vert \alpha \right\vert +\left\vert
\beta \right\vert =M}\vec{a}_{\alpha \beta }\left( x,y\right) \cdot \partial
_{x}^{\alpha }\partial _{y}^{\beta }\vec{P}+L_{low}\vec{P},  \label{sec4-4}
\end{equation}%
where we write $L_{low}$ to denote a differential operator with polynomial
coefficients of order $<M$. The symbol $L_{low}$ will be used to denote
several such operators, i.e., $L_{low}$ may denote different operators in
different occurrences.

In (\ref{sec4-4}), $\vec{a}_{\alpha \beta }$ is a vector, and each component
of $\vec{a}_{\alpha \beta }$ is a polynomial in $x_{1},\cdots
,x_{n},y_{1},\cdots ,y_{m}$. Also, in (\ref{sec4-4}), we use
\textquotedblleft $\cdot $\textquotedblright\ to denote the dot product of
vectors.

From (\ref{sec4-4}), we obtain 
\begin{eqnarray*}
&&\left[ \Delta \left( x,y\right) \right] ^{M}L\left( \vec{P}\right) \\
&=&\sum_{\left\vert \alpha \right\vert +\left\vert \beta \right\vert
=M,\alpha =\left( \alpha _{1},\cdots ,\alpha _{n}\right) }\vec{a}_{\alpha
\beta }\left( x,y\right) \left[ \Delta \left( x,y\right) \right]
^{\left\vert \beta \right\vert }\cdot \left( \Delta \left( x,y\right) \partial
_{x_{1}}\right) ^{\alpha _{1}}\cdots \left( \Delta \left( x,y\right)
\partial _{x_{n}}\right) ^{\alpha _{n}}\partial _{y}^{\beta }\vec{P} \\
&&+L_{low}\vec{P} \\
&=&\sum_{\left\vert \left( \alpha _{1},\cdots ,\alpha _{n}\right)
\right\vert +\left\vert \beta \right\vert =M}\vec{a}_{\alpha \beta }\left(
x,y\right) \left[ \Delta \left( x,y\right) \right] ^{\left\vert \beta
\right\vert }\cdot \left( X_{1}-Y_{1}\right) ^{\alpha _{1}}\cdots \left(
X_{n}-Y_{n}\right) ^{\alpha _{n}}\partial _{y}^{\beta }\vec{P}+L_{low}\vec{P}%
.
\end{eqnarray*}%
We express each $\left( X_{1}-Y_{1}\right) ^{\alpha _{1}}\cdots \left(
X_{n}-Y_{n}\right) ^{\alpha _{n}}$ as a sum of terms, each of which is a
product of factors $X_{j}$ or $Y_{j}$.

Modulo operators $L_{low}$, we may move all $X_{j}$ in each term to the left
of all the $Y_{j}$ in that term. Therefore, 
\begin{eqnarray}
&&\left[ \Delta \left( x,y\right) \right] ^{M}L\left( \vec{P}\right)  \notag
\\
&=&\sum_{\left\vert \gamma \right\vert +a+b=M}\sum_{\substack{ j_{1}\leq
j_{2}\leq \cdots \leq j_{a}  \\ k_{1}\leq k_{2}\leq \cdots \leq k_{b}}}\vec{%
\phi}_{k_{1}\cdots k_{b},\gamma }^{j_{1}\cdots j_{a}}\cdot
X_{j_{1}}X_{j_{2}}\cdots X_{j_{a}}Y_{k_{1}}Y_{k_{2}}\cdots Y_{k_{b}}\partial
_{y}^{\gamma }\vec{P}+L_{low}\left( \vec{P}\right) ,  \label{sec4-5}
\end{eqnarray}%
with the $\vec{\phi}$'s (vector valued) polynomials in $x_{1},\cdots ,x_{n}$,
$y_{1},\cdots ,y_{m}$.

Let $D$ be a large enough integer constant, to be picked below.

Then (\ref{sec4-5}) yields 
\begin{eqnarray}
&&\left[ \Delta \left( x,y\right) \right] ^{D}L\left( \vec{P}\right)  \notag
\\
&=&\sum_{\left\vert \gamma \right\vert +a+b=M}\sum_{\substack{ j_{1}\leq
j_{2}\leq \cdots \leq j_{a}  \\ k_{1}\leq k_{2}\leq \cdots \leq k_{b}}}%
X_{j_{1}}\cdots X_{j_{a}}\left[ \Delta \left( x,y\right) \right] ^{D-M}\vec{%
\phi}_{k_{1}\cdots k_{b},\gamma }^{j_{1}\cdots j_{a}}\cdot Y_{k_{1}}\cdots
Y_{k_{b}}\partial _{y}^{\gamma }\vec{P}  \label{sec4-6} \\
&&+\left[ \Delta \left( x,y\right) \right] ^{D-2M}L_{low}\left( \vec{P}%
\right),  \notag
\end{eqnarray}%
thanks to \eqref{xxxx1}.

Now $L_{low}$ in (\ref{sec4-6}) satisfies all the assumptions made on $L$,
except that $L_{low}$ has order $\leq M-1$, whereas $L$ has order $\leq M$.
Therefore, just as we proved (\ref{sec4-6}) for $L$, we can now prove 
\begin{eqnarray}
&&\left[ \Delta \left( x,y\right) \right] ^{D-2M}L_{low}\left( \vec{P}\right)
\notag \\
&=&\sum_{\left\vert \gamma \right\vert +a+b=M-1}\sum_{\substack{ j_{1}\leq
\cdots \leq j_{a}  \\ k_{1}\leq \cdots \leq k_{b}}}X_{j_{1}}\cdots X_{j_{a}}%
\left[ \Delta \left( x,y\right) \right] ^{\left( D-2M\right) -\left(
M-1\right) }\vec{\psi}_{k_{1}\cdots k_{b},\gamma }^{j_{1}\cdots j_{a}}\cdot
Y_{k_{1}}\cdots Y_{k_{b}}\partial _{y}^{\gamma }\vec{P}  \label{sec4-7} \\
&&+\left[ \Delta \left( x,y\right) \right] ^{\left( D-2M\right) -2\left(
M-1\right) }L_{lower}\left( \vec{P}\right) .  \notag
\end{eqnarray}%
Here, the $\vec{\psi}$'s are (vector-valued) polynomials in $x_{1},\cdots
,x_{n},y_{1},\cdots ,y_{m}$, and $L_{lower}$ is a linear differential
operator of order $\leq M-2$ with polynomial coefficients in $\mathbb{R}%
\left[ x_{1},\cdots ,x_{n},y_{1},\cdots ,y_{m}\right] $. Also, we point out
that $L_{low}$ in (\ref{sec4-7}) is the same as $L_{low}$ in (\ref{sec4-6}).

Again, there is a formula for $\left[ \Delta \left( x,y\right) \right]
^{\left( D-2M\right) -2\left( M-1\right) }$ $L_{lower}\left( \vec{P}\right)
, $ analogous to (\ref{sec4-6}) and (\ref{sec4-7}).

Continuing in this way, we eventually reach $L_{lowest}$, a differential
operator of order $0$.

We simply pick $D$ large enough so that all the powers of $\Delta \left(
x,y\right) $ appearing in the above formulas (\ref{sec4-6}), (\ref{sec4-7}) $%
\cdots $ are non-negative.

From the formulas (\ref{sec4-6}), (\ref{sec4-7}) $\cdots $ , we conclude
that 
\begin{eqnarray}
&&\left[ \Delta \left( x,y\right) \right] ^{D}L\left( \vec{P}\right)  \notag
\\
&=&\sum_{\left\vert \gamma \right\vert +a+b\leq M}\sum_{\substack{ j_{1}\leq
\cdots \leq j_{a}  \\ k_{1}\leq \cdots \leq k_{b}}}X_{j_{1}}\cdots X_{j_{a}}%
\vec{\theta}_{k_{1}\cdots k_{b},\gamma }^{j_{1}\cdots j_{a}}\cdot
Y_{k_{1}}\cdots Y_{k_{b}}\partial ^{\gamma}_y\vec{P}  \label{sec4-8}
\end{eqnarray}%
where the coefficients $\vec{\theta}_{k_{1}\cdots k_{b},\gamma
}^{j_{1}\cdots j_{a}}$ are vectors whose components are polynomials in $%
x_{1},\cdots ,x_{n}$, $y_{1},\cdots ,y_{m}$.

Moreover, everything we did in deriving (\ref{sec4-8}) was effective; hence,
we may compute the integer $D$ and the polynomials $\vec{\theta}%
_{k_{1}\cdots k_{b},\gamma }^{j_{1}\cdots j_{a}}$ in (\ref{sec4-8}).

It is convenient to express (\ref{sec4-8}) in a different notation. First of
all, for fixed $j_{1}\leq \cdots \leq j_{a}$, we define 
\begin{eqnarray*}
&&L_{j_{1}\cdots j_{a}}\left( \vec{P}\right) \\
&=&\sum_{\left\vert \gamma \right\vert +b\leq M-a}\sum_{k_{1}\leq \cdots
\leq k_{b}}\vec{\theta}_{k_{1}\cdots k_{b},\gamma }^{j_{1}\cdots j_{a}}\cdot
Y_{k_{1}}\cdots Y_{k_{b}}\partial ^{\gamma }_y\vec{P}.
\end{eqnarray*}%
Thus, $L_{j_{1}\cdots j_{a}}$ is a linear differential operator involving no 
$x$-derivatives and having coefficients in $\mathbb{R}\left[ x_{1},\cdots
,x_{n},y_{1},\cdots ,y_{m}\right] $. Thus, (\ref{sec4-8}) becomes%
\begin{eqnarray*}
&&\left[ \Delta \left( x,y\right) \right] ^{D}L\left( \vec{P}\right) \\
&=&\sum_{a \leq M } \sum_{j_1\leq j_2 \leq \cdots \leq j_a} X_{j_{1}}\cdots
X_{j_{a}}L_{j_{1}\cdots j_{a}}\left( \vec{P}\right) .
\end{eqnarray*}%
For a multiindex $\alpha =\left( \alpha _{1},\cdots ,\alpha _{n}\right) $,
we define $X^{\alpha }=X_{1}^{\alpha _{1}}\cdots X_{n}^{\alpha _{n}}$. (The $X_j$ needn't commute.)

Then the above formula for $[\Delta (x,y)]^{D}L(\vec{P})$ takes the form%
\begin{equation}
\left[ \Delta \left( x,y\right) \right] ^{D}L\left( \vec{P}\right)
=\sum_{\left\vert \alpha \right\vert \leq M}X^{\alpha }L_{\alpha }\left( 
\vec{P}\right) ,  \label{sec4-9}
\end{equation}%
where each $L_{\alpha }$ is a linear differential operator involving no $x$%
-derivatives and having coefficients in $\mathbb{R}\left[ x_{1},\cdots
,x_{n},y_{1},\cdots ,y_{m}\right] $.

Moreover, 
\begin{equation}
\text{we can compute }D\text{ and all the }L_{\alpha }\text{ from the data
provided at the beginning of the section.}  \label{sec4-10}
\end{equation}

We now prove the following

{\textbf{Claim. }}{\textit{Let $\vec{P}=(P_{1},\cdots ,P_{J})$ with each $%
P_{j} \in \mathbb{R}[x_1, \cdots, x_n,y_1,\cdots, y_m]$. Then%
\begin{equation}
L\left( Q\vec{P}\right) =0\text{ on }V\text{ for all }Q\in \mathbb{R}\left[
x_{1},\cdots ,x_{n},y_{1},\cdots ,y_{m}\right]  \label{sec4-11}
\end{equation}%
if and only if 
\begin{equation}
L_{\alpha }\left( Q\vec{P}\right) =0\text{ on }V\text{ for all }Q\in \mathbb{%
R}\left[ x_{1},\cdots ,x_{n},y_{1},\cdots ,y_{m}\right] \text{, }\left\vert
\alpha \right\vert \leq M.  \label{sec4-12}
\end{equation}%
}}

Once we have established the claim, we can simply take $L_{1},\cdots ,L_{N}$
in Algorithm \ref{main-algorithm-section4} to be a list of all the $%
L_{\alpha }$. Thus, to complete our explanation of that algorithm, it is
enough to prove the above claim.

First suppose $\vec{P}$ satisfies $L_{\alpha }\left( Q\vec{P}\right) =0$ on $%
V$ for all $Q$ and all $\left\vert \alpha \right\vert \leq M$. Recalling
from (\ref{sec4-3}) that each $X_{j}$ is tangent to $V$, and recalling the
definition $X^{\alpha }=X_{1}^{\alpha _{1}}\cdots X_{n}^{\alpha _{n}}$ for $%
\alpha =\left( \alpha _{1},\cdots ,\alpha _{n}\right) $, we conclude that $%
X^{\alpha }L_{\alpha }\left( Q\vec{P}\right) =0$ on $V$ for all $Q$, and for
all $\left\vert \alpha \right\vert \leq M$. Summing over $\alpha $ and
recalling (\ref{sec4-9}), we conclude that 
\begin{equation*}
\left[ \Delta \left( x,y\right) \right] ^{D}L\left( Q\vec{P}\right) =0\text{
on }V\text{ for all }Q\text{.}
\end{equation*}%
We have seen that $V\setminus \left\{ \left( x,y\right) \in V:\Delta
\left( x,y\right) =0\right\} $ is dense in $V$. Therefore, for each $Q$, $L\left( Q\vec{P}\right) =0$
on a dense subset of $V$. Since $Q$, the components of $\vec{P}$, and the
coefficients of $L$ are all polynomials in $x_{1},\cdots ,x_{n},y_{1},\cdots
,y_{m}$, we conclude that $L\left( Q\vec{P}\right) =0$ on $V$. Thus, we have
proven that (\ref{sec4-12}) implies (\ref{sec4-11}).

Next, we show that (\ref{sec4-11}) implies (\ref{sec4-12}).

For all $M_{\ast }\leq M$, we prove that 
\begin{equation}
\sum_{\left\vert \alpha \right\vert \leq M_{\ast }}X^{\alpha }\left(
L_{\alpha }\left( Q\vec{P}\right) \right) =0\text{ on }V\text{ for all }Q
\label{sec4-13}
\end{equation}%
implies%
\begin{equation}
L_{\alpha }\left( Q\vec{P}\right) =0\text{ on }V\text{ for all }Q\text{ and
all }\left\vert \alpha \right\vert \leq M_{\ast }.  \label{sec4-14}
\end{equation}%
Once we know that (\ref{sec4-13}) implies \eqref{sec4-14}, we just take $%
M_{\ast}=M$; then \eqref{sec4-13} follows from \eqref{sec4-11} , \eqref{sec4-14} becomes \eqref{sec4-12}, and the proof of
Claim will be complete. Thus, it remains to prove that \eqref{sec4-13}
implies \eqref{sec4-14}. We proceed by induction on $M_{\ast} $.

For $M_{\ast}=0$, \eqref{sec4-13} and \eqref{sec4-14} are obviously
equivalent, since there is only one multiindex $\alpha$ with $|\alpha|\leq 0$%
.

For the induction step, fix $M_{\ast} \geq 1$, and suppose that \eqref{sec4-13}$\implies$\eqref{sec4-14} holds with $M_{\ast}-1$ in place of 
$M_{\ast}$. We will prove \eqref{sec4-13}$\implies$\eqref{sec4-14} for the
given $M_{\ast}$.

In fact, suppose $\vec{P}$ satisfies \eqref{sec4-13}. Fix $(x^{0},y^{0})\in
V $, and $|\alpha |=M_{\ast }$. Recalling the form \eqref{sec4-2} of $X_{j}$%
, we see that%
\begin{equation*}
X^{\beta }=\left( \Delta \left( x,y\right) \right) ^{\left\vert \beta
\right\vert }\partial _{x}^{\beta }+L_{\beta }^{err},
\end{equation*}%
where $L_{\beta }^{err}$ differentiates at most $\left\vert \beta
\right\vert -1$ times in $x$, and perhaps many times in $y$. Therefore, for
any smooth function $A\left( x,y\right) $, we have 
\begin{equation}
X^{\beta }\left[ \left( x-x^{0}\right)^\alpha A\left( x,y\right) \right]
|_{\left( x^{0},y^{0}\right) }=0\text{ for }\left\vert \beta \right\vert
<\left\vert \alpha \right\vert =M_{\ast },  \label{sec4-15}
\end{equation}%
and 
\begin{equation}
X^{\beta }\left[ \left( x-x^{0}\right)^\alpha A\left( x,y\right) \right]
|_{\left( x^{0},y^{0}\right) }=\alpha !\delta _{\beta \alpha }\left( \Delta
\left( x^{0},y^{0}\right) \right) ^{M_{\ast }}\cdot A\left(
x^{0},y^{0}\right) \text{ for }\left\vert \beta \right\vert =\left\vert
\alpha \right\vert =M_{\ast }.  \label{sec4-16}
\end{equation}%
For any $Q\in \mathbb{R}\left[ x_{1},\cdots ,x_{n},y_{1},\cdots ,y_{m}\right]
$, we have from (\ref{sec4-13})%
\begin{equation*}
\sum_{\left\vert \beta \right\vert \leq M_{\ast }}X^{\beta }\left( L_{\beta
}\left( \left( x-x^{0}\right)^\alpha Q\cdot \vec{P}\right) \right) =0\text{
at }\left( x^{0},y^{0}\right) .
\end{equation*}%
However, $L_{\beta }$ involves no $x$-derivatives, and therefore 
\begin{equation*}
L_{\beta }\left( \left( x-x^{0}\right) ^{\alpha }Q\cdot \vec{P}\right)
=\left( x-x^{0}\right) ^{\alpha }L_{\beta }\left( Q\cdot \vec{P}\right) .
\end{equation*}%
Hence, 
\begin{equation*}
\sum_{\left\vert \beta \right\vert \leq M_{\ast }}X^{\beta }\left[ \left(
x-x^{0}\right) ^{\alpha }L_{\beta }\left( Q\cdot \vec{P}\right) \right] =0%
\text{ at }\left( x^{0},y^{0}\right) .
\end{equation*}%
Taking $A\left( x,y\right) =L_{\beta }\left( Q\cdot \vec{P}\right) $ in (\ref%
{sec4-15}), (\ref{sec4-16}), we now conclude that 
\begin{equation*}
\left( \Delta \left( x^{0},y^{0}\right) \right) ^{M_{\ast }}L_{\alpha
}\left( Q\cdot \vec{P}\right) =0\text{ at }\left( x^{0},y^{0}\right) .
\end{equation*}%
This holds for any $Q\in \mathbb{R}\left[ x_{1},\cdots ,x_{n},y_{1},\cdots
,y_{m}\right] $, any $\left( x^{0},y^{0}\right) \in V$ and any $\left\vert
\alpha \right\vert =M_{\ast }$. Thus, $\left( \Delta \left( x,y\right)
\right) ^{M_{\ast }}L_{\alpha }\left( Q\cdot \vec{P}\right) =0$ on $V$, for $%
Q\in \mathbb{R}\left[ x_{1},\cdots ,x_{n},y_{1},\cdots ,y_{m}\right] $ and $%
\left\vert \alpha \right\vert =M_{\ast }$.

Recalling that $V\setminus \{(x,y)\in V:\Delta (x,y)=0\}$ is
dense in $V$, we now conclude that 
\begin{equation}
L_{\alpha }(Q\cdot \vec{P})=0\text{ on }V  \label{sec4-17}
\end{equation}
for all $Q\in \mathbb{R}\left[ x_{1},\cdots ,x_{n},y_{1},\cdots ,y_{m}\right]
$ and $\left\vert \alpha \right\vert =M_{\ast }$. Recalling that the $X_{j}$
are tangent to $V$, we conclude that 
\begin{equation*}
X^{\alpha }\left( L_{\alpha }\left( Q\cdot \vec{P}\right) \right) =0\text{
on }V\text{ for all }Q,\text{ and for }\left\vert \alpha \right\vert
=M_{\ast }\text{.}
\end{equation*}%
Together with (\ref{sec4-13}), this tells us that 
\begin{equation*}
\sum_{\left\vert \alpha \right\vert \leq M_{\ast }-1}X^{\alpha }\left(
L_{\alpha }\left( Q\cdot \vec{P}\right) \right) =0\text{ on }V
\end{equation*}%
for all $Q\in \mathbb{R}\left[ x_{1},\cdots ,x_{n},y_{1},\cdots ,y_{m}\right]
$. Hence, by induction hypothesis (\eqref{sec4-13}$\implies $\eqref{sec4-14}
with $M_{\ast }$ replaced by $M_{\ast }-1$), we have 
\begin{equation*}
L_{\alpha }\left( Q\cdot \vec{P}\right) =0\text{ on }V\text{ for all }Q\text{
and }\left\vert \alpha \right\vert \leq M_{\ast }-1\text{.}
\end{equation*}%
Together with (\ref{sec4-17}) this completes the proof that \eqref{sec4-13}$%
\implies $\eqref{sec4-14} for the given $M_{\ast }$. This also completes our
induction on $M_{\ast }$, the proof of Claim, and the explanation of
Algorithm \ref{main-algorithm-section4}.
\end{proof}

\section{Solutions of Differential Equations IV\label{section-5-solutionofDE}%
}

In this section, we suppose we are given a semialgebraic, connected,
real-analytic submanifold%
\begin{equation*}
V\subset \mathbb{R}^{n}\times \mathbb{R}^{m}\times \mathbb{R}^{p}\text{.}
\end{equation*}%
We write $\left( x,y,z\right) $ to denote a point of $\mathbb{R}^{n}\times 
\mathbb{R}^{m}\times \mathbb{R}^{p}$, with $x=\left( x_{1},\cdots
,x_{n}\right) $, $y=\left( y_{1},\cdots ,y_{m}\right) $, and $z=\left(
z_{1},\cdots ,z_{p}\right) $.

We assume that $x_{1},\cdots ,x_{n}$ serve as local coordinates in some
nonempty relatively open subset of $V$.

We suppose we are given nonzero polynomials $\tilde{P}_{\mu }\left( x,y_{\mu
}\right) $ ($\mu =1,\cdots ,m$) and $\hat{P}_{\lambda }\left( x,z_{\lambda
}\right) $ ($\lambda =1,\cdots p$) that vanish on $V$.

We write $\mathbb{R}\left[ x,y,z\right] $ to denote the ring $\mathbb{R}%
\left[ x_{1},\cdots ,x_{n},y_{1},\cdots ,y_{m},z_{1},\cdots ,z_{p}\right] $;
similarly for $\mathbb{R}\left[ x,y\right] $. We write $\vec{P}$ to denote a
vector $\left( P_{1},\cdots ,P_{J}\right) $ with components in $\mathbb{R}%
\left[ x,y,z\right] $, and we write $\vec{P}\left( x,y\right) $ to denote a
vector $\left( P_{1},\cdots ,P_{J}\right) $ with components in $\mathbb{R}%
\left[ x,y\right] $.

We suppose we are given a linear differential operator $L$ with polynomial
coefficients in $\mathbb{R}\left[ x,y,z\right] $. The operator $L$ acts on
vectors $\vec{P}$ and produces scalars (i.e., $L\vec{P}$ has just one
component).

We present the following

\begin{algorithm}
\label{find-generators-formodules-pde-v1}Given $V$, $\tilde{P}_{\mu }$, $%
\hat{P}_{\lambda }$, $L$, we find generators for the $\mathbb{R}\left[ x,y%
\right] $-module of all polynomial vectors $\vec{P}\left( x,y\right) $ such
that $L\left( Q\cdot \vec{P}\right) =0$ on $V$ for all $Q\in \mathbb{R}\left[
x,y,z\right] $.
\end{algorithm}

\begin{proof}[\textbf{Explanation}]
Regarding $z_{1},\cdots ,z_{p}$ as $y_{m+1},\cdots ,y_{m+p}$, we apply
Algorithm \ref{main-algorithm-section4}, to produce linear differential
operators $L_{1},\cdots ,L_{N}$ with the following properties.

\begin{itemize}
\item Each $L_{\nu }$ maps vectors $\vec{P}$ to scalars (i.e., $L_{\nu }%
\vec{P}$ has only one component).

\item Each $L_{\nu }$ has polynomial coefficients in $\mathbb{R}[x,y,z]$.

\item Each $L_{\nu }$ involves only $y$ and $z$-derivatives (no $x$%
-derivatives).

\item For any $\vec{P}$, we have $L(Q\vec{P})=0$ on $V$ for all $Q\in 
\mathbb{R}[x,y,z]$ if and only if $L_{\nu }(Q\vec{P})=0$ on $V$ for all $%
Q\in \mathbb{R}[x,y,z]$ and all $\nu =1,\cdots ,N$.
\end{itemize}

For each $\nu =1,\cdots ,N$, we apply Algorithm \ref{main-algorithm-section3}
to produce generators for the $\mathbb{R}[x,y]$-module%
\begin{equation*}
\mathscr{M}_{\nu }=\left\{ \vec{P}\left( x,y\right) :L_{\nu }\left( Q\vec{P}%
\right) =0\text{ on }V\text{ for all }Q\in \mathbb{R}\left[ x,y,z\right]
\right\} \text{.}
\end{equation*}%
Thanks to the last bullet point above, the $\mathbb{R}\left[ x,y\right] $-module 
\begin{equation*}
\mathscr{M}=\left\{ \vec{P}\left( x,y\right) :L\left( Q\vec{P}\right) =0%
\text{ on }V\text{ for all }Q\in \mathbb{R}\left[ x,y,z\right] \right\}
\end{equation*}%
is equal to $\mathscr{M}_{1}\cap \cdots \cap \mathscr{M}_{N}$. Since we have
computed generators for each $\mathscr{M}_{\nu }$, we can compute generators
for their intersection. This completes our explanation of the algorithm.
\end{proof}

\section{Solutions of Differential Equations V\label{section7}}

In this section, we work in $\mathbb{R}^{n}$ and we write $x=\left(
x_{1},\cdots ,x_{n}\right) $ to denote a point of $\mathbb{R}^n$. We fix $J\geq 1$, and we write $\vec{F}=\left(
F_{1},\cdots ,F_{J}\right) $ to denote a vector of smooth functions on $%
\mathbb{R}^{n}$.

We suppose we are given a linear differential operator $L$, acting on
vectors $\vec{F}$, and producing scalar-valued functions $L\vec{F}$.

We assume that $L$ has semialgebraic coefficients.

\begin{algorithm}[Main Algorithm for Differential Equations]
\label{main-algorithm-for-differential-equations}Given $L$ as above, we
compute generators for the $\mathbb{R}\left[ x_{1},\cdots ,x_{n}\right] $%
-module $\mathcal{M}(L)$ consisting of all polynomial vectors $\vec{P}=\left(
P_{1},\cdots ,P_{J}\right) $ such that $L\left( Q\vec{P}\right) =0$ on $%
\mathbb{R}^{n}$ for all $Q\in \mathbb{R}\left[ x_{1},\cdots ,x_{n}\right] $.
\end{algorithm}

\begin{proof}[\textbf{Explanation}]
Let $A_{1}\left( x\right) $, $A_{2}\left( x\right) ,\cdots ,A_{K}\left(
x\right) $ be a list of all the coefficients of $L$. The $A_{k}$ are
semialgebraic functions. 

Applying Algorithm \ref{algorithm-6.5} to the semialgebraic set $E = \mathbb{%
R}^n$ and the list of functions $A_1,\cdots, A_K$, we obtain the following:

\begin{itemize}
\item A partition of $\mathbb{R}^n$ into finitely many semialgebraic sets $%
E_\nu \quad (\nu =1 ,\cdots, N)$.

\item For each $\nu$, an invertible linear map $T_\nu: \mathbb{R}^n
\rightarrow \mathbb{R}^n$.

\item We guarantee that, for each $\nu $, $T_{\nu }E_{\nu }$ has the form

\begin{itemize}
\item $T_{\nu }E_{\nu }=\left\{ \left( x^{\prime },x^{\prime \prime }\right)
\in \mathbb{R}^{n_{\nu }}\times \mathbb{R}^{n-n_{\nu }}:x^{\prime }\in
U_{\nu },x^{\prime \prime }=G_{\nu }\left( x^{\prime }\right) \right\} $,
where

\item $U_{\nu }\subset \mathbb{R}^{n_{\nu }}$ is an open, connected
semialgebraic set and

\item $G_{\nu }:U_{\nu }\rightarrow \mathbb{R}^{n-n_{\nu }}$ is
real-analytic and semialgebraic.

\item Moreover, for each $\lambda =1,\cdots ,K$, we have $A_{\lambda }\circ
T_{\nu }^{-1}\left( x^{\prime },G_{\nu }\left( x^{\prime }\right) \right)
=H_{\lambda \nu }\left( x^{\prime }\right) $ for all $x^{\prime }\in U_{\nu
} $, where $H_{\lambda \nu }:U_{\nu }\rightarrow \mathbb{R}$ is a
real-analytic semialgebraic function.
\end{itemize}
\end{itemize}

Algorithm \ref{algorithm-6.5} computes the above objects $E_{\nu },T_{\nu
},U_{\nu },G_{\nu },H_{\lambda \nu }$ as well as nonzero polynomials

\begin{itemize}
\item $P_{\mu \nu }\left( x_{1}^{\prime },\cdots ,x_{n_{\nu }}^{\prime
},y_{\mu }\right) \quad \left( \mu =1,\cdots ,n-n_{\nu },\nu =1,\cdots
,N\right) $ and

\item $\hat{P}_{\lambda \nu }\left( x_{1}^{\prime },\cdots ,x_{n_{\nu
}}^{\prime },z_{\lambda }\right) \quad \left( \lambda =1,\cdots ,K,\nu
=1,\cdots ,N\right) $, such that%
\begin{equation*}
P_{\mu \nu }\left( x_{1}^{\prime },\cdots ,x_{n_{\nu }}^{\prime },y_{\mu
}\right) =0\text{ whenever }x^{\prime }=\left( x_{1}^{\prime },\cdots
,x_{n_{\nu }}^{\prime }\right) \in U_{\nu }\text{ and }\left( y_{1},\cdots
,y_{n-n_{\nu }}\right) =G_{\nu }\left( x^{\prime }\right);
\end{equation*}%
and 
\begin{equation*}
\hat{P}_{\lambda \nu }\left( x_{1}^{\prime },\cdots ,x_{n_{\nu }}^{\prime
},H_{\lambda \nu }\left( x_{1}^{\prime },\cdots ,x_{n_{\nu }}^{\prime
}\right) \right) =0\text{ for all }\left( x_{1}^{\prime },\cdots ,x_{n_{\nu
}}^{\prime }\right) \in U_{\nu }\text{.}
\end{equation*}
\end{itemize}

We define $L_{\nu }\vec{F}=\left[ L\left( \vec{F}\circ T_{\nu }\right) %
\right] \circ T_{\nu }^{-1}$. Thus, $L\vec{F}=0$ if and only if $L_{\nu
}\left( F\circ T_{\nu }^{-1}\right) =0$ on $T_{\nu }\left( E_{\nu }\right) $
for each $\nu =1,\cdots ,N$.

Therefore our module $\mathcal{M}(L)$ is equal to the intersection
over all $\nu =1,\cdots N$ of the $\mathbb{R}\left[ x_{1},\cdots ,x_{n}%
\right] $-modules 
\begin{equation*}
\mathcal{\mathcal{M}}_{\nu }=\left\{ \vec{P}\circ T_{\nu }:\vec{P}\in 
\mathcal{M}_{\nu }^{\#}\right\} \text{, where }
\end{equation*}

\begin{itemize}
\item[\refstepcounter{equation}\text{(\theequation)}\label{7.3.1}] $\mathcal{%
M}_{\nu }^{\#}$ is the $\mathbb{R}\left[ x_{1},\cdots ,x_{n}\right] $-module
consisting of all polynomial vectors $\vec{P}=\left( P_{1},\cdots
,P_{J}\right) $ such that 
\begin{equation*}
L_{\nu }\left( Q\cdot \vec{P}\right) =0
\end{equation*}%
on $\Gamma_{\nu }=\left\{ \left( x^{\prime },x^{\prime \prime }\right) \in 
\mathbb{R}^{n_{\nu }}\times \mathbb{R}^{n-n_{\nu }}:x^{\prime }\in U_{\nu
},x^{\prime \prime }\in G_{\nu }\left( x^{\prime }\right) \right\} $ for all 
$Q \in \mathbb{R}[x^{\prime },x^{\prime \prime }]$.
\end{itemize}

We will compute generators for each $\mathcal{M}_\nu^{\#}$. From those, we
can compute generators for $\mathcal{M}_\nu$, then for $\mathcal{M}$.

The coefficients of $L$ are the functions $A_1,\cdots, A_K$; and $L_\nu$
arose from $L$ by a known linear coordinate change $T_\nu$.

\begin{itemize}
\item[\refstepcounter{equation}\text{(\theequation)}\label{7.3.2}] %
Therefore, the coefficients of $L_{\nu }$ are equal on $\Gamma_{\nu }$ to known linear
combinations of the functions $H_{\lambda \nu }(x^{\prime })\quad (\lambda
=1,\cdots ,K)$.
\end{itemize}

We now introduce the semialgebraic set

\begin{itemize}
\item[\refstepcounter{equation}\text{(\theequation)}\label{7.3.3}] $V_{\nu
}=\left\{ \left( x^{\prime },x^{\prime \prime },z_{1},\cdots ,z_{K}\right)
:x^{\prime }\in U_{\nu },x^{\prime \prime }=G_{\nu }\left( x^{\prime
}\right) ,z_{\lambda }=H_{\lambda \nu }\left( x^{\prime }\right) \text{ for }%
\lambda =1,\cdots ,K\right\} $.
\end{itemize}

Since $U_{\nu }$ is connected, open and semialgebraic in $\mathbb{R}^{n_{\nu
}}$, and since $G_{\nu },H_{\lambda \nu }$ are real-analytic and
semialgebraic on $U_{\nu }$, we see that $V_{\nu }$ is a semialgebraic,
connected, real-analytic manifold in $\mathbb{R}^{n_{\nu }}\times \mathbb{R}%
^{n-n_{\nu }}\times \mathbb{R}^{K}$ , on which the rectangular coordinates $%
x_1^{\prime}, \cdots, x_{n_\nu}^{\prime}$ of $x^{\prime }$ serve as global
real-analytic coordinates.\newline
Moreover, our nonzero polynomials $P_{\mu \nu },\hat{P}_{\lambda \nu }$
satisfy 
\begin{equation*}
P_{\mu \nu }\left( x_{1}^{\prime },\cdots ,x_{n_{\nu }}^{\prime },y_{\mu
}\right) =0
\end{equation*}
for $\left( x_{1}^{\prime },\cdots ,x_{n_{\nu }}^{\prime },y_{1},\cdots
,y_{n-n_{\nu }},z_{1},\cdots ,z_{k}\right) \in V_{\nu }$ \ and 
\begin{equation*}
\hat{P}_{\lambda \nu }\left( x_{1}^{\prime },\cdots ,x_{n_{\nu }}^{\prime
},z_{\lambda }\right) =0
\end{equation*}%
for $\left( x_{1}^{\prime },\cdots ,x_{n_{\nu }}^{\prime },y_{1},\cdots
,y_{n-n_{\nu }},z_{1},\cdots ,z_{k}\right) \in V_{\nu }$.\newline
Thus, the $V_{\nu },P_{\mu \nu },\hat{P}_{\lambda \nu }$ (for fixed $\nu $)
are as in the setup for Section \ref{section-5-solutionofDE}.

Moreover, we may lift $L_{\nu }$ in a natural way to a differential operator 
$L_{\nu }^{\#}$, acting on vectors of smooth functions on $\mathbb{R}%
^{n_{\nu }}\times \mathbb{R}^{n-n_{\nu }}\times \mathbb{R}^{K}$. To define $%
L_{\nu }^{\#}$, and to see its relationship to $L_{\nu }$, we recall the
remark (\ref{7.3.2}). Thus, for a finite list of coefficients $\left( \Omega
_{\alpha \beta \nu}^{\lambda j}\right) $, we have 
\begin{equation*}
L_{\nu }\left( F_{1},\cdots ,F_{J}\right) \left( x^{\prime },x^{\prime
\prime }\right) =\sum_{\lambda ,j,\alpha ,\beta }\Omega _{\alpha \beta \nu
}^{\lambda j}\cdot H_{\lambda \nu }\left( x^{\prime }\right) \cdot \partial
_{x^{\prime }}^{\alpha }\partial _{x^{\prime \prime }}^{\beta }F_{j}\left(
x^{\prime },x^{\prime \prime }\right)
\end{equation*}%
for $\left( x^{\prime },x^{\prime \prime }\right) \in \Gamma_{\nu }$. The $\Omega _{\alpha \beta
\nu}^{\lambda j}$ are real numbers, which we can compute. 

Recalling \eqref{7.3.3}, we are led to define%
\begin{equation}
L_{\nu }^{\#}\left( F_{1},\cdots ,F_{J}\right) \left( x^{\prime },x^{\prime
\prime },z_{1},\cdots ,z_{K}\right) =\sum_{\lambda ,j,\alpha ,\beta }\Omega
_{\alpha \beta 
\nu}^{\lambda j}\cdot z_\lambda \cdot \partial _{x^{\prime
}}^{\alpha }\partial _{x^{\prime \prime }}^{\beta }F_j\left( x^{\prime
},x^{\prime \prime },z_{1},\cdots ,z_{K}\right)  \label{7.3.star}
\end{equation}%
for vectors $(F_1,\cdots,F_J)$ of smooth functions defined on $\mathbb{R}^{n_{\nu }}\times 
\mathbb{R}^{n-n_{\nu }}\times \mathbb{R}^{K}$. \newline
The relationship between $L_{\nu }$ and $L_{\nu }^{\#}$ is as follows:
Suppose $\vec{F}=\left( F_{1},\cdots ,F_{J}\right) $, where the
functions $F_{j}:\mathbb{R}^{n_{\nu }}\times \mathbb{R}^{n-n_{\nu }}\times 
\mathbb{R}^{K}\rightarrow \mathbb{R}$ do not depend on the last $K$
coordinates. Then we may regard $\vec{F}$ either as a vector-valued function
on $\mathbb{R}^{n_{\nu }}\times \mathbb{R}^{n-n_{\nu }}\times \mathbb{R}^{K}$
or on $\mathbb{R}^{n_{\nu }}\times \mathbb{R}^{n-n_{\nu }}$. We have $L_{\nu
}\vec{F}=0$ on $\Gamma_{\nu }$ if and only if $L_{\nu }^{\#}\vec{F}=0$ on $V_{\nu
}$.

Accordingly, our definition \eqref{7.3.1} of the module $\mathcal{M}_{\nu
}^{\#}$ is equivalent to the following

\begin{itemize}
\item[\refstepcounter{equation}\text{(\theequation)}\label{7.3.4}] $\mathcal{%
M}_{\nu }^{\#}$ is the $\mathbb{R}[x_{1}^{\prime },\cdots ,x_{n_{\nu
}}^{\prime },y_{1},\cdots ,y_{n-n_{\nu }}]$-module of all polynomial vectors 
$$\vec{P}=(P_{1},\cdots ,P_{J}),$$ with each $P_{j}\in \mathbb{R}%
[x_{1}^{\prime },\cdots ,x_{n_{\nu }}^{\prime },y_{1},\cdots ,y_{n-n_{\nu
}}] $, such that 
\begin{equation*}
L_{\nu }^{\#}\left( Q\cdot \vec{P}\right) =0\text{ on }V_{\nu }
\end{equation*}%
for all $Q\in \mathbb{R}[x_{1}^{\prime },\cdots ,x_{n_{\nu }}^{\prime
},y_{1},\cdots ,y_{n-n_{\nu }}]$.
\end{itemize}

Note that $L_{\nu }^{\#}$ is a linear differential operator with polynomial
coefficients in 
\begin{equation*}
\mathbb{R}[x_{1}^{\prime },\cdots ,x_{n_{\nu }}^{\prime },y_{1},\cdots
,y_{n-n_{\nu }},z_{1},\cdots ,z_{K}].
\end{equation*}

Thus, $L_\nu^{\#}$ is as in the setup for Section \ref%
{section-5-solutionofDE}.

We are now in position to apply Algorithm \ref%
{find-generators-formodules-pde-v1}. Thus, we compute generators of the 
\begin{equation*}
\mathbb{R}[x_{1}^{\prime },\cdots ,x_{n_{\nu }}^{\prime },y_{1},\cdots
,y_{n-n_{\nu }}]\text{-module}
\end{equation*}

\begin{itemize}
\item[\refstepcounter{equation}\text{(\theequation)}\label{7.3.5}] $\mathcal{%
M}_{\nu}^{\# \#}$ consisting of all polynomial vectors $\vec{P}%
=(P_1,\cdots,P_J)$ with each 
\begin{equation*}
P_j \in \mathbb{R}[x_{1}^{\prime },\cdots ,x_{n_{\nu }}^{\prime
},y_{1},\cdots ,y_{n-n_{\nu }}],
\end{equation*}
such that $L_\nu^{\#}(Q\vec{P})=0$ on $V_\nu$ for all $Q \in \mathbb{R}%
[x_{1}^{\prime },\cdots ,x_{n_{\nu }}^{\prime },y_{1},\cdots ,y_{n-n_{\nu
}},z_1,\cdots,z_K] $.
\end{itemize}

Note that the $Q$ allowed in \eqref{7.3.5} are more general than the $Q$
allowed in \eqref{7.3.4}.

We now show that

\begin{itemize}
\item[\refstepcounter{equation}\text{(\theequation)}\label{7.3.6}] $\mathcal{%
M}_{\nu}^{\#}=\mathcal{M}_{\nu}^{\#\#}$.
\end{itemize}

Once we know this, we will have computed generators for each $\mathcal{M}_{\nu }^{\#}$%
; as we explained earlier, this allows us to compute generators for the
module $\mathcal{M}(L)$ introduced in the statement of Algorithm \ref%
{main-algorithm-for-differential-equations}. Thus, our task is reduced to
proving \eqref{7.3.6}.

Trivially, $\mathcal{M}_{\nu }^{\#\#}\subset \mathcal{M}_{\nu }^{\#}$. Our
task is to show that $\mathcal{M}_{\nu }^{\#}\subset \mathcal{M}_{\nu
}^{\#\#}$.

Let $\vec{P}=\left( P_{1},\cdots ,P_{J}\right) $ belong to $\mathcal{M}_{\nu
}^{\#}$, and let $\left( x^{o},y^{o},z^{o}\right) \in V$, where $%
x^{o}=\left( x_{1}^{o},\cdots ,x_{n_{\nu }}^{o}\right) $, $y^{o}=\left(
y_{1}^{o},\cdots ,y_{n-n_{\nu }}^{o}\right) $, $z^{o}=\left(
z_{1}^{o},\cdots ,z_{K}^{o}\right) $.

Thus, each $P_{j}$ is a polynomial in the variables $x_{1},\cdots ,x_{n_{\nu
}}$, $y_{1},\cdots ,y_{n-n_{\nu }}$ (not involving $z_{1},\cdots ,z_{K}$),
and we have 
\begin{equation*}
L_{\nu }^{\#}\left( Q^{o}\cdot \vec{P}\right) |_{\left(
x^{o},y^{o},z^{o}\right) }=0
\end{equation*}%
for any $Q^{o}\in \mathbb{R}\left[ x_{1},\cdots ,x_{n_{\nu }},y_{1},\cdots
,y_{n-n_{\nu }}\right] $.

Let $Q\in \mathbb{R}[x_{1},\cdots ,x_{n},y_{1},\cdots ,y_{n-n_{\nu
}},z_{1},\cdots ,z_{K}]$ be given. Define 
\begin{equation*}
Q^{o}\left( x_{1},\cdots ,x_{n},y_{1},\cdots ,y_{n-n_{\nu }}\right) =Q\left(
x_{1},\cdots ,x_{n},y_{1},\cdots ,y_{n-n_{\nu }},z_{1}^{o},\cdots
,z_{K}^{o}\right) \text{.}
\end{equation*}%
Thus%
\begin{equation*}
Q^{o}\in \mathbb{R}\left[ x_{1},\cdots ,x_{n_{\nu }},y_{1},\cdots
,y_{n-n_{\nu }}\right] \text{,}
\end{equation*}%
hence 
\begin{equation*}
L_{\nu }^{\#}\left( Q^{o}\cdot \vec{P}\right) |_{\left(
x^{o},y^{o},z^{o}\right) }=0\text{.}
\end{equation*}

However, since $L_{\nu }^{\#}$ (see \eqref{7.3.star}) involves no
derivatives in the $z_{1},\cdots ,z_{K}$, we have 
\begin{equation*}
L_{\nu }^{\#}\left( Q\cdot \vec{P}\right) |_{\left( x^{o},y^{o},z^{o}\right)
}=L_{\nu }^{\#}\left( Q^{o}\cdot \vec{P}\right) |_{\left(
x^{o},y^{o},z^{o}\right) }\text{.}
\end{equation*}

Therefore, $L_{\nu }^{\#}\left( Q\cdot \vec{P}\right) |_{\left(
x^{o},y^{o},z^{o}\right) }=0$.

Since $\left( x^{o},y^{o},z^{o}\right) \in V$ and $Q\in \mathbb{R}\left[
x_{1},\cdots ,x_{n_{\nu }},y_{1},\cdots ,y_{n-n_{\nu }},z_{1},\cdots ,z_{K}%
\right] $ are arbitrary, we conclude from \eqref{7.3.5} that $P\in \mathcal{M%
}_{\nu }^{\#\#}$.

Thus $\mathcal{M}_{\nu }^{\#}\subset \mathcal{M}_{\nu }^{\#\#}$, proving %
\eqref{7.3.6} and completing the explanation of Algorithm \ref%
{main-algorithm-for-differential-equations}.
\end{proof}

\bibliographystyle{plain}
\bibliography{papers}

\end{document}